\documentclass[review,hidelinks,onefignum,onetabnum]{siamart251216}

\usepackage{mathbbol}
\usepackage{accents}

\usepackage{mathtools}
\newcommand{\mynegspace}{\hspace{-0.12em}}
\newcommand{\lvvvert}{\rvert\mynegspace\rvert\mynegspace\rvert}
\DeclarePairedDelimiter{\vvvert}{\lvvvert}{\lvvvert}
\graphicspath{{Figures/}}

\usepackage{cancel}
\usepackage{color}
\usepackage{fancyhdr}
\usepackage{graphics}
\usepackage{hyperref}
\usepackage{graphicx,epsf}
\usepackage{makeidx}
\usepackage{epsfig}
\usepackage{lscape}
\usepackage{empheq,xcolor}

\usepackage{lipsum}
\usepackage{amsfonts}
\usepackage{graphicx}
\usepackage{epstopdf}
\usepackage{algorithmic}
\ifpdf
  \DeclareGraphicsExtensions{.eps,.pdf,.png,.jpg}
\else
  \DeclareGraphicsExtensions{.eps}
\fi

\newsiamremark{remark}{Remark}
\newsiamremark{hypothesis}{Hypothesis}
\crefname{hypothesis}{Hypothesis}{Hypotheses}
\newsiamthm{claim}{Claim}
\newsiamremark{fact}{Fact}
\crefname{fact}{Fact}{Facts}

\headers{ETD schemes for Multicomponent Membranes}{Wangbo Luo, Zhonghua Qiao, and Yanxiang Zhao}

\title{Exponential Time Differencing Schemes for a Phase-Field Model of Multicomponent Membranes\thanks{
Submitted to the editors DATE. 
\funding{The first and second authors were partially supported by NSFC/RGC Joint Research Scheme grant N\_PolyU5145/24 and Hong Kong Research Grants Council GRF grant 15305624.}
}
}

\author{Wangbo Luo\thanks{Department of Applied Mathematics, The Hong Kong Polytechnic University, Hung Hom, Kowloon, Hong Kong
  (\email{wangbo.luo@polyu.edu.hk}, \email{zhonghua.qiao@polyu.edu.hk}).}
  \and Zhonghua Qiao\footnotemark[2]
\and Yanxiang Zhao\thanks{Corresponding author. Department of Mathematics, George Washington University, Washington, DC, USA
  (\email{yxzhao@gwu.edu}).}
}

\usepackage{amsopn}

\ifpdf
\hypersetup{
  pdftitle={Exponential Time Differencing Schemes for a Phase-Field Model of Multicomponent Membranes},
  pdfauthor={Wangbo Luo, Zhonghua Qiao, and Yanxiang Zhao}
}
\fi

\begin{document}

\maketitle

\begin{abstract}
In this paper, we develop and analyze exponential time differencing (ETD) schemes for a phase-field model of multicomponent membranes proposed in our previous work \cite{luo2025ohta}, in which membrane deformation is governed by a force-balance phase-field equation and protein segregation is described by a membrane-associated Ohta-Kawasaki (OK) dynamics. For a fixed phase-field membrane, we introduce a geometry-adapted operator splitting method based on the localization function, which reformulates the surface OK dynamics into a form suitable for ETD integration. The resulting first- and second-order ETD schemes, combined with finite-difference spatial discretization, are rigorously proved to satisfy a discrete maximum-bound principle and unconditional energy stability. For the coupled system, we construct stabilized ETD schemes in an FFT-based spectral framework, treating stiff linear terms exactly and nonlinear mechanochemical couplings explicitly. A narrow-band implementation further reduces the computational cost by restricting surface calculations to the diffuse membrane region. Numerical experiments confirm the predicted temporal accuracy, maximum-bound preservation, and energy decay for the fixed-membrane OK problem, and demonstrate stable and efficient three-dimensional simulations of protein-driven pattern formation and membrane deformation.
\end{abstract}

\begin{keywords}
Multicomponent membranes, 
Ohta–Kawasaki model, Exponential time differencing schemes, Maximum bound principle, Energy stability
\end{keywords}

\begin{MSCcodes}
65D99, 65L20, 65M70
\end{MSCcodes}

\section{Introduction}\label{Sec:Intro}
Multicomponent membranes, composed of lipid bilayers and associated proteins, play important roles in many cellular processes, such as cellular transport, signaling, adhesion, endocytosis, vesicle trafficking, and cell division \cite{schmid2017physical,yuan2021membrane,tsai2024study}. Experiments on giant unilamellar vesicles (GUVs) show that lipid-protein mixtures can separate into coexisting domains and generate diverse membrane shapes, including buds, tubes, and finger-like protrusions \cite{dietrich2001lipid,baumgart2003imaging,su2024kinetic}. These results indicate a strong coupling between membrane composition and morphology, and motivate computational models that integrate phase separation with membrane mechanics. Over the past decades, phase field methods have become one of the most useful tools for modeling and simulating moving interfaces in biological microstructures. These methods avoid explicit interface tracking and naturally describe large deformations and topological changes. They have been successfully applied to lipid bilayer bending, vesicle shape evolution, and membrane dynamics \cite{du2004phase,du2005modeling,wang2008modelling}, as well as to cell crawling, polarity, and chemotaxis through coupled reaction--diffusion and force-balance systems \cite{shao2010computational,camley2013periodic,camley2014polarity,camley2017crawling,zhang2025phase,zhang2012reaction}. These works show that phase field methods provide an effective approach for studying biological systems with coupled geometry, mechanics, and biochemical regulation.

Recently, we proposed a diffuse-interface mechanochemical model for multicomponent membranes in which membrane-bound proteins reorganize on the surface and drive membrane deformation \cite{luo2025ohta}. The model uses two phase-field variables: $\phi$ describes the vesicle membrane as a diffuse interface separating the interior and exterior regions, while $u$ represents the membrane-associated protein density, distinguishing protein-rich and protein-poor phases. Membrane deformation is governed by a force-balance equation that combines bending, surface tension, area constraint, line tension, and protein-dependent biochemical forces. Protein segregation is described by the Ohta-Kawasaki (OK) type dynamics restricted to the membrane surface. The resulting coupled system is given by \cite{luo2025ohta}:
\begin{align}
    & \mu\frac{\partial\phi}{\partial t} = \lambda_{\mathrm{surf}}\left(\Delta\phi-\frac{1}{\epsilon_{\phi}^2}W'(\phi)\right) - \kappa\left(\Delta - \frac{W''(\phi)}{\epsilon_{\phi}^2}\right)\left(\Delta\phi - \frac{W'(\phi)}{\epsilon_{\phi}^2}\right) \nonumber\\
    & \quad\quad\quad + \lambda_{\mathrm{line}} \left(\epsilon_u \Delta_{\mathrm{S}}u - \frac{1}{\epsilon_u}g(\phi)W'(u)\right)|\nabla\phi| - M_{\mathrm{area}}\left(\int_{\Omega}\phi \, dx - A_0\right)|\nabla\phi| \nonumber \\ 
    & \quad\quad\quad + \alpha (u + u_0)\left(\epsilon_{\phi}\Delta\phi-\frac{1}{\epsilon_{\phi}}W'(\phi)\right)|\nabla\phi|, \label{eqn:model_force} \\
    & \frac{\partial(g(\phi)u)}{\partial t} + \nabla\cdot(g(\phi)u\mathbf{v}) = \epsilon_u \Delta_{\mathrm{S}}u  - \frac{1}{\epsilon_u} g(\phi)W'(u) + \gamma g(\phi)\Delta_{\mathrm{S}}^{-1}\Big(g(\phi)(u-\bar{u})\Big) \nonumber \\  
    &\quad\quad\quad\quad\quad\quad\quad\quad\quad\quad\quad\quad - M g(\phi) \left(\int_{\Omega}g(\phi)\left(u-\bar{u}\right)\, dx \right).
    \label{eqn:model_OK}
\end{align}
Here, $W(\phi)=18(\phi^2-\phi)^2$ is the double-well potential, which enforces $\phi\approx 1$ inside the vesicle and $\phi\approx 0$ outside. In \eqref{eqn:model_force}, the first two terms describe surface tension and bending elasticity, while the remaining terms account for protein-induced line tension, area conservation, and biochemical forcing. The parameters $\lambda_{\mathrm{surf}}$, $\kappa$, $\lambda_{\mathrm{line}}$, $M_{\mathrm{area}}$, and $\alpha$ determine the strengths of these effects, and $\epsilon_{\phi}$ controls the membrane thickness. Equation \eqref{eqn:model_OK} describes membrane-bound protein segregation through the OK-type dynamics. The localization function $g(\phi)=\epsilon_{\phi}^{-1}W(\phi)$ restricts the evolution of $u$ to the diffuse membrane, and the surface diffusion operator is defined by $\Delta_{\mathrm{S}}:= \nabla\cdot(g(\phi)\nabla)$. The parameter $\epsilon_u$ sets the width of the protein interface, while the nonlocal term involving $\Delta_{\mathrm{S}}^{-1}$ encodes long-range interaction with the parameter $\gamma$ and promotes microphase separation. The last term enforces a soft mass constraint around the prescribed average protein density $\bar{u}$. The advection term $\nabla\cdot(g(\phi)u\mathbf{v})$ couples protein transport to the membrane velocity $\mathbf{v}= -\partial_t\phi\frac{\nabla\phi}{|\nabla\phi|^2}$. Together, \eqref{eqn:model_force}--\eqref{eqn:model_OK} couple curvature elasticity, protein phase separation, and biochemical forcing in a diffuse-interface system for multicomponent membranes.

Many phase-field models in cell biology are based on force balance, coupling membrane mechanics, cytoskeletal forces, adhesion, and biochemical regulation into PDE systems for cell shape, polarity, and migration
\cite{luo2025ohta,shao2010computational,camley2013periodic,camley2014polarity,camley2017crawling,zhang2025phase,zhang2012reaction,ziebert2012model,zimmermann2010leading}. These models can reproduce experimental behaviors through nonlinear mechanochemical coupling, but their numerical studies remain limited to forward Euler or first-order semi-implicit time discretizations.
These approaches usually require small time steps for stability, lose accuracy over long time intervals, and become prohibitively expensive for biologically relevant simulations and parameter studies. In addition, rigorous analysis for convergence, stability, and energy dissipation law is still limited for many of these schemes. Thus, there is a clear need for stable, accurate, and efficient methods for current mechanochemical models.

Exponential time differencing (ETD) methods are well-suited for stiff differential equations by treating the linear part exactly in time and handling the nonlinear terms explicitly \cite{HoOs10}. In recent years, they have been systematically developed and analyzed for a range of phase-field models
\cite{wang2016efficient,du2019maximum,du2021maximum,wang2025maximum,ju2018energy}, with discrete maximum bound principle (MBP), energy stability, and related error estimates. They can be reduced to multiplying matrix functions by vectors, which can be implemented efficiently using fast Fourier transforms (FFT) \cite{wang2016efficient,du2019maximum} or other fast algorithms for structured matrices \cite{ju2015fast}. These features make ETD methods suitable for large-scale simulations of coupled phase-field models in cell biology and materials science.

A key component of \eqref{eqn:model_force}--\eqref{eqn:model_OK} is the membrane-associated OK model for protein phase separation. The OK model, first introduced in \cite{ohta1986equilibrium}, has been widely used to describe microphase separation in diblock copolymer melts \cite{hamley2004developments} and thus is well-suited to simulate the phase separation of the membrane-associated proteins. For OK-type energies on regular domains, many analytical results have been established on properties of global minimizers \cite{Choksi2012,Ren_Shoup2020,LuoZhao_PhysicaD2024}. Numerically, Allen-Cahn-type OK dynamics have been studied with backward differentiation formula (BDF) methods \cite{Xu_Zhao2019,Xu_Zhao2020,LuoZhao_NumPDE2024,LuoZhao_AAMM2024,luo2026double} and ETD methods \cite{wang2025maximum}, with rigorous analysis of MBP and energy stability. Cahn-Hilliard-type OK dynamics have also been explored using implicit midpoint spectral schemes \cite{Benesova_SINA2014}, IEQ methods \cite{ChengYangShen_JCP2017}, and stabilized semi-implicit schemes \cite{XuTang_SJNA2006}. However, there is still a lack of studies of OK dynamics on complex and deformable geometries, especially surfaces represented by phase-field functions. By incorporating a localization function $g(\phi)$, the membrane-associated OK energy is given by \cite{luo2025ohta}:
\begin{align}
    E_{\phi}[u] :=&
    \int_{\Omega} g(\phi)
    \left[
    \frac{\epsilon_u}{2}|\nabla u|^2
    +\frac{1}{\epsilon_u}W(u)
    \right]\,\mathrm{d}\mathbf{x}
    \nonumber\\
    &+\frac{\gamma}{2}
    \int_{\Omega}
    \left[
    (-\Delta_{\mathrm S})^{-1/2}
    \big(g(\phi)(f(u)-\bar u)\big)
    \right]^2
    \,\mathrm{d}\mathbf{x}
    +\frac{M}{2}
    \left[
    \int_{\Omega} g(\phi)(f(u)-\bar u)\,\mathrm{d}\mathbf{x}
    \right]^2 ,
    \label{eqn:mem_OK_energy}
\end{align}
where $\Omega=\prod_{i=1}^{d}[-X_i,X_i]\subset\mathbb{R}^d$, $d=2,3$. The new term $f(u) = 3u^2 - 2u^3$ is a smooth indicator of the protein-rich phase \cite{Xu_Zhao2020}. The evolution equation \eqref{eqn:model_OK} can then be viewed as the $L^2$-gradient flow of \eqref{eqn:mem_OK_energy}, with an additional advection term induced by membrane motion.

The first part of this work studies the membrane-associated OK dynamics on a fixed diffuse membrane. We prescribe the phase-field function $\phi$, so that the localization function $g(\phi)$ is time independent, the membrane velocity satisfies $\mathbf v=0$, and the advection term is omitted. The protein variable $u$ then evolves according to the $L^2$-gradient flow of $E_\phi[u]$:
\begin{align}\label{eqn:pACOK_mem}
    \frac{\partial(g(\phi)u)}{\partial t}
    =
    -\frac{\delta E_{\phi}[u]}{\delta u}
    ={}& \epsilon_u \Delta_{\mathrm{S}}u
    - \gamma g(\phi)f'(u)(-\Delta_{\mathrm{S}})^{-1}
    \Big(g(\phi)(f(u)-\bar{u})\Big) \nonumber \\
    &-\frac{1}{\epsilon_u}g(\phi)W'(u)
    -M g(\phi)f'(u)
    \left(\int_{\Omega}g(\phi)(f(u)-\bar{u})\,\mathrm{d}x\right).
\end{align}
Here, $f(u)=3u^2-2u^3$ satisfies $f(0)=0$, $f(1)=1$, and $f'(0)=f'(1)=0$, which are key properties for proving the discrete MBP for BDF and ETD methods \cite{Xu_Zhao2020,wang2025maximum}. Since $g(\phi)>0$ is fixed, the gradient-flow structure gives the continuous energy dissipation law as
\begin{align*}
    \frac{\mathrm{d}}{\mathrm{d}t}E_{\phi}[u]
    =
    \left(\frac{\delta E_{\phi}[u]}{\delta u},\frac{\partial u}{\partial t}\right)_{\Omega}
    =
    -\left(\frac{\delta E_{\phi}[u]}{\delta u},
    g(\phi)^{-1}\frac{\delta E_{\phi}[u]}{\delta u}\right)_{\Omega}
    \leq 0 .
\end{align*}
For this fixed-geometry problem, we construct first- (ETD1) and second-order ETD Runge-Kutta (ETDRK2) schemes, combined with second-order finite differences and a narrow-band implementation near the diffuse membrane. We prove a discrete MBP and energy stability for both schemes. To our knowledge, these are the first proposed and rigorously analyzed ETD schemes for a membrane-associated OK model posed on a phase-field-defined curved geometry.

The second part of this work develops ETD schemes for the coupled system \eqref{eqn:model_force}-\eqref{eqn:model_OK}, where both the membrane shape and the membrane-bound protein distribution evolve in time. In these schemes, the linear parts of the force-balance phase-field equation \eqref{eqn:model_force} and the membrane-associated OK equation \eqref{eqn:model_OK} are treated exactly, combined with the spectral collocation approximation and the operator-splitting stabilization \cite{wang2016efficient} for the force-balance equation, together with another stabilization strategy for the membrane-associated OK dynamics. The nonlinear terms are handled explicitly and are discretized in space with methods tailored to their structure: spectral collocation for the force-balance phase-field equation, and finite difference methods (with narrow-band restriction) for the membrane-associated OK equation. This mixed spatial treatment keeps the linear solves simple and fast in Fourier space while preserving flexibility for the surface differential equation. Although a full convergence and energy analysis of the coupled ETD schemes is still open, extensive numerical experiments show stable, accurate, and efficient long-time simulations of multicomponent membrane dynamics.

The main contributions of this paper are summarized as follows:
\begin{enumerate}
\item We formulate the membrane-associated OK dynamics on a fixed phase-field-defined membrane as an $L^2$-gradient flow and derive its energy dissipation law. Based on a new operator splitting method adapted to the diffuse-surface geometry, we construct ETD1 and ETDRK2 schemes and prove their discrete MBP and energy stability.
\item We develop ETD1 and ETDRK2 schemes for the mechanochemical phase-field system. The proposed algorithms combine spectral collocation and finite differences in a unified framework for the coupled system, together with suitable stabilization techniques and a narrow-band strategy for efficient long-time simulations.
 \item The proposed ETD framework provides an efficient numerical approach for phase-field systems on evolving membranes and can be extended to related force-balance models for cell shape, polarity, migration, and chemotaxis. 
\end{enumerate}

\section{Exponential Time Differencing Schemes for Membrane–associated OK Model on Fixed Membranes}\label{Sec:ETDforMOK}
In this section, we present the fully discrete ETD1 and ETDRK2 schemes for the membrane-associated OK model on a fixed phase-field cell $\phi$, and then prove their discrete  
maximum bound principle (MBP) and energy stability for the membrane-based OK energy functional \eqref{eqn:mem_OK_energy}.

Since the cell $\phi$ is fixed, the localization function $g(\phi)$ is also fixed and independent of time, and the advection term in \eqref{eqn:model_OK} 
can be neglected. Thus, the corresponding equation degenerates into
\begin{align}
         \frac{\partial(g(\phi)u)}{\partial t} =  \ & \epsilon_{u} \Delta_{\mathrm{S}}u  - \frac{1}{\epsilon_{u}} g(\phi)W'(u) - \gamma g(\phi)f'(u)(-\Delta_{\mathrm{S}})^{-1}\Big(g(\phi)(f(u)-\bar{u})\Big) \nonumber \\  
         & - M g(\phi)f'(u) \left(\int_{\Omega}g(\phi)\left(f(u)-\bar{u}\right)\text{d}x \right). \label{eqn:model_OK_fix}
\end{align}

\subsection{Second-order Finite Difference Scheme for Spatial Discretization}
In this section, we discretize the spatial operators by using the finite difference method. Consider a square domain $\Omega = [-X, X)\times[-Y, Y)\subset\mathbb{R}^2$, and a periodic boundary condition is imposed for the problem.  For the three-dimensional case, the notations can be defined similarly. Let $N_x$ and $N_y$ be two positive even integers, $h_x = \frac{2X}{N_x}$ and $h_y = \frac{2Y}{N_y}$, and discretize the spatial domain $\Omega$ by a uniform mesh as $\Omega_h = \Omega\cap (h_x\mathbb{Z}\times h_y\mathbb{Z})$. We then define the index set 
\begin{align}
    S_{h} = \{ (i,j)\in\mathbb{Z}^2;\ i = 1:N_x,\ j = 1:N_y\}.
\end{align}
Denote by $\mathcal{M}_{h}$ the collection of periodic grid functions defined on $\Omega_{h}$:
\begin{align*}
    \mathcal{M}_{h} = \{f:\Omega_{h}\to \mathbb{R} | \ f_{i+m_1N_x,j+m_2N_y} = f_{ij}, \ \forall\ (i,j)\in S_{h}, \ \forall\ (m_1,m_2) \in\mathbb{Z}^2\}.
\end{align*}
and $\accentset{\circ}{\mathcal{M}}_{h} := \{f\in\mathcal{M}_{h}|\left\langle f,1\right\rangle_{h} =0 \}$ the collection of functions in ${\mathcal{M}}_{h}$ with zero mean. For any $f,g\in \mathcal{M}_{h}$ and $\mathbf{f} = (f^1,f^2)^T, \mathbf{g} = (g^1, g^2)^T \in( \mathcal{M}_{h})^2$, we define the discrete $L^2$ inner product $\left\langle\cdot,\cdot\right\rangle_{h}$, discrete $L^2$-norm $\|\cdot\|_{L^2,h}$, and discrete $L^{\infty}$-norm $\|\cdot\|_{L^{\infty},h}$ as follows:
\begin{align*}
   & \langle f,g \rangle_{h} = h_xh_y\sum_{(i,j)\in S_h}f_{ij}g_{ij}, \ \ \|f\|_{L^2,h} = \sqrt{\langle f,f \rangle_{h}}, \ \ \|f\|_{L^{\infty},h} = \max_{(i,j)\in S_h}|f_{ij}|, \\
   &  \langle \mathbf{f},\mathbf{g} \rangle_{h} =  h_xh_y\sum_{(i,j)\in S_h}\left(f^1_{ij}g^1_{ij}+f^2_{ij}g^2_{ij}\right), \ \ ||\mathbf{f}||_{L^2,h} = \sqrt{\langle \mathbf{f},\mathbf{f} \rangle_{h}}.
\end{align*}
We further define the discrete operator $\Delta_{\mathrm{S},h}$ as 
\begin{align}
   &  (\Delta_{\mathrm{S},h}u)_{ij} = D_x^{-}\left(g_{i+\frac{1}{2},j}D_x^{+}u_{ij}\right) + D_y^{-}\left(g_{i,j+\frac{1}{2}}D_y^{+}u_{ij}\right), \label{operator:diff} 
\end{align}
where 
\begin{align*}
    & D_x^{+}u_{ij} = \frac{u_{i+1,j}-u_{ij}}{h_x}, \quad D_y^{+}u_{ij} = \frac{u_{i,j+1}-u_{ij}}{h_y}, \\
    & D_x^{-}u_{ij} = \frac{u_{ij}-u_{i-1,j}}{h_x}, \quad D_y^{-}u_{ij} = \frac{u_{ij}-u_{i,j-1}}{h_y}, \\
    & g_{i+\frac{1}{2},j} = \frac{g_{i+1,j} + g_{ij}}{2}, \quad g_{i,j+\frac{1}{2}} = \frac{g_{i,j + 1} + g_{ij}}{2}.
\end{align*}
and periodic boundary conditions apply when $i\notin\{1,\cdots,N_x\}$ and $j\notin\{1,\cdots,N_y\}$.

Furthermore, to analyze the numerical schemes, we take the fixed phase-field cell $\phi$ that satisfies $0<\phi<1$, which leads to the positivity of the localization function $0<g_{\mathrm{min}}\leq g(\phi)\leq g_{\mathrm{max}}$, and is required in the following proofs.

For any $g>0$, it is obvious that $-\Delta_{\mathrm{S},h}$ is symmetric and positive semidefinite on $\accentset{\circ}{\mathcal{M}}_{h}$. Then it's safe to define its inverse $(-\Delta_{\mathrm{S},h})^{-1}:\accentset{\circ}{\mathcal{M}}_{h}\to\accentset{\circ}{\mathcal{M}}_{h}$ by
\begin{align}
    (-\Delta_{\mathrm{S},h})^{-1}f = u \quad \mathrm{if\ and\ only\ if} \quad -\Delta_{\mathrm{S},h}u = f. \nonumber
\end{align}
The definition of $(-\Delta_{\mathrm{S},h})^{-1}$ can be extended to the functions with nonzero mean by removing the mean value, namely, $(-\Delta_{\mathrm{S},h})^{-1}f = (-\Delta_{\mathrm{S},h})^{-1}\left(f-\frac{1}{\Omega_{h}}\langle f,1 \rangle_{h}\right)$ for any $f\in\mathcal{M}_{h}$. Then, to prove the discrete MBP and energy stability, we denote the $L^{\infty}$-bound $\|(-\Delta_{\mathrm{S},h})^{-1}\|_{L^{\infty},h}$ as the optimal constant $C$ such that $\|(-\Delta_{\mathrm{S},h})^{-1}f\|_{L^{\infty},h} \leq C\|f\|_{L^{\infty},h}$. In the following lemmas, we prove the $L^{\infty}$-bound of the discrete inverse operator $(-\Delta_{\mathrm{S},h})^{-1}$.

\begin{lemma}[$L^{\infty}$-bound of $(-\Delta_{\mathrm{S},h})^{-1}$]\label{lemma:Linfbound}
    For any functions $u,f \in \accentset{\circ}{\mathcal{M}}_{h}$ such that $-\Delta_{\mathrm{S},h}u = f$, the $L^{\infty}$-bound of the discrete inverse operator is given by 
    \begin{align}
        \|(-\Delta_{\mathrm{S},h})^{-1}\|_{L^{\infty},h}\leq C, \nonumber
    \end{align}
    where $C$ is a generic constant independent of $h_x$, $h_y$. 
\end{lemma}


\begin{proof}
By the construction of $-\Delta_{\mathrm{S},h}$ in \eqref{operator:diff} and the
uniform bound $0<g_{\mathrm{min}}\leq g(\phi)\leq g_{\mathrm{max}}$, $-\Delta_{\mathrm{S},h}$ is a uniformly elliptic finite difference operator in divergence form and is symmetric positive definite on $\accentset{\circ}{\mathcal M}_h$. 

Consider the discrete Green's function for $-\Delta_{\mathrm{S},h}$ on $\accentset{\circ}{\mathcal M}_h$. For each grid point $x_{ij}\in\Omega_h$, let
$G_h(x_{ij},\cdot)\in \accentset{\circ}{\mathcal M}_h$ solve
\begin{align*}
        -\Delta_{\mathrm{S},h} G_h(x_{ij},\cdot)
    =
    \delta_{x_{ij},h}
    -
    \frac{1}{|\Omega|}\mathbb 1,
\end{align*}
where $ \delta_{x_{ij},h} $ is the discrete delta function
\begin{align*}
    \delta_{x_{ij},h}(x_{pq}) = 
    \begin{cases}
        \frac{1}{h_x h_y}, \quad (p,q)=(i,j), \\
        0, \quad \mathrm{otherwise}
    \end{cases}
\end{align*}
Then, the Green's function representation formula gives
\begin{align*}
        u_{ij}
    =
    \left\langle G_h(x_{ij},\cdot),f\right\rangle_h
    =
    \sum_{x_{pq}\in\Omega_h}
    G_h(x_{ij},x_{pq}) f_{pq}\,h_xh_y .
\end{align*}
By the discrete Green's function estimate for uniformly elliptic finite-difference operators in divergence form on a periodic grid,
\cite{BrambleHubbardThomee1969,BrambleThomee1969,BealeLayton2006}
\[
    \sup_{x_{ij}\in\Omega_h}
    \sum_{x_{pq}\in\Omega_h}
    \left|G_h(x_{ij},x_{pq})\right|h_xh_y
    \leq C,
\]
where $C$ depends on $\Omega, g_{\mathrm{max}}$ and $g_{\mathrm{min}}$ due to the uniform ellipticity bounds of $-\Delta_{\mathrm{S},h}$, but not on $h_x$ and $h_y$. Therefore, for each $x_{ij}\in\Omega_h$,
\begin{align*}
    |u_{ij}|
    \leq
    \sum_{x_{pq}\in\Omega_h}
    \left|G_h(x_{ij},x_{pq})\right|
    |f_{pq}|\,h_xh_y .
\end{align*}
Taking the maximum over $(i,j)$ gives
\begin{align*}
    \|u\|_{L^\infty,h} = \|(-\Delta_{\mathrm{S},h})^{-1}f\|_{L^{\infty},h}
    \leq
    C\|f\|_{L^\infty,h}.
\end{align*}
\end{proof}

\subsection{ETD1 and ETDRK2 Schemes with Maximum Bound Principle and Energy Stability}

Instead of working on the original equation \eqref{eqn:model_OK_fix} for the discrete MBP and energy stability, we construct a new operator splitting method by dividing $g^{1/2}$ on both sides of the equation, which leads to 
\begin{align}
         \frac{\partial(g^{1/2}u)}{\partial t} =  \ & \epsilon_{u} \left(g^{-1/2}\Delta_{\mathrm{S}}g^{-1/2}\right)g^{1/2}u  - \gamma g^{1/2}f'(u)(-\Delta_{\mathrm{S}})
         ^{-1}\Big(g(\phi)(f(u)-\bar{u})\Big) \nonumber \\  
         & - \frac{1}{\epsilon_{u}} g^{1/2}W'(u) - M g^{1/2}f'(u) \left(\int_{\Omega}g(\phi)\left(f(u)-\bar{u}\right) \text{d}x \right). \label{eqn:model_OK_half}
\end{align}
Introducing the linear operator and the nonlinear remainder as 
\begin{align*}
     \mathcal{L}_{\phi} = \ & B - \epsilon_{u} g^{-1/2}\Delta_{\mathrm{S}}g^{-1/2} , \mathrm{\ namely,\ } \mathcal{L}_{\phi}w = B w - \epsilon_{u} g^{-1/2}\Delta_{\mathrm{S}}(g^{-1/2}w), \\
     \mathcal{R}_{\phi}(u) = \ &  Bg^{1/2}u  - \frac{1}{\epsilon_{u}}g^{1/2}W'(u) - \gamma g^{1/2}(-\Delta_{\mathrm{S}})^{-1}\left(g(\phi)f(u)\right)f'(u) \\
     & - Mg^{1/2}\int_{\Omega}g(\phi)(f(u)-\bar{u})\ \mathrm{d}xf'(u),
\end{align*}
where $\mathcal{L}_{\phi}$ and $\mathcal{R}_{\phi}$ depend on the phase-field cell $\phi$ through $g = 18(\phi-\phi^2)^2$, and $B$ is a stabilization constant. The equation \eqref{eqn:model_OK_half} can be reformulated as
\begin{align}
\frac{\partial (\sqrt{g}u)}{\partial t} + \mathcal{L}_{\phi}(\sqrt{g}u) = \mathcal{R}_{\phi}(u), \label{eqn:OK_FD_operatorsplit}
\end{align}
and the corresponding spatial-discrete scheme is to find a function $u:[0,\infty)\to\mathcal{M}_{h}$ with a fixed $g = (g_{ij})> 0 \in \mathcal{M}_{h}$ such that 
\begin{align}
    & \frac{\partial (\sqrt{g}\odot u)}{\partial t} + \mathbf{L}_{\phi}(\sqrt{g}\odot u) = \mathbf{R}_\phi(u), \label{eqn:ODE_eqn} \\
    & u(0) = u_0, \label{eqn:ODE_ini}
\end{align}
where $\odot$ is entry-wise multiplication, $u_0$ is the given initial data, and 
\begin{align}
     &\mathbf{L}_{\phi} = \ BI - \epsilon_{u} g^{-\frac{1}{2}}\odot\Delta_{\mathrm{S},h}(g^{-\frac{1}{2}}\odot), \nonumber\\
     &\mathbf{R}_\phi(u) =\ B \sqrt{g} \odot u - \epsilon_{u}^{-1} \sqrt{g}\odot W'(u) - \gamma \sqrt{g}\odot(-\Delta_{\mathrm{S},h})^{-1}\Big(g\odot f(u))\Big)\odot f'(u)\nonumber\\
     &\qquad\qquad\  - M \Big\langle g\odot (f(u)-\bar{u}),\mathbb{1}\Big\rangle_h \sqrt{g}\odot f'(u). \nonumber
\end{align}
Note that the discrete operator $\Delta_{\mathrm{S},h}$ is symmetric, negative semi-definite, and $g > 0$; thus, $\mathbf{L}_{\phi}$ is also symmetric and positive definite for any $B>0$. 

\begin{remark}\label{remark:idtf}
In fact, we can identify the grid function in $\mathcal M_h$ with its coefficient vector in $\mathbb R^{N_xN_y}$ by arranging the grid values in lexicographical order. Under this identification, every discrete linear operator on $\mathcal M_h$ is represented uniquely by a matrix in $\mathbb R^{N_xN_y\times N_xN_y}$. In particular, the discrete operator $\Delta_{\mathrm{S},h}$ is equivalent to a stiffness matrix, which we still denote by $\Delta_{\mathrm{S},h}$ for notational simplicity. Likewise, we take
\begin{align*}
    G=\mathrm{diag}(g_{ij})\in\mathbb R^{N_xN_y\times N_xN_y},
\end{align*}
and $I_{N_xN_y}$ denotes the identity matrix in $\mathbb R^{N_xN_y\times N_xN_y}$, then the discrete operator $\mathbf L_{\phi}$ can be represented by
\[
\mathbf L_{\phi}=BI_{N_xN_y}-\epsilon_u G^{-1/2}\Delta_{\mathrm{S},h}G^{-1/2}.
\]
Therefore, the action of the discrete operator on a grid function is exactly the matrix-vector multiplication of its coefficient vector, and we do not need to distinguish the linear operators and corresponding matrices in the following analysis. 
\end{remark}

Next, we give the positivity property of the matrix exponential
generated by $\mathbf L_{\phi}$. This property follows from the
standard theory of Metzler matrices; see \cite{farina2011positive}.

\begin{lemma}\label{lemma:LuPD}
    Given the operator $\mathbf{L}_{\phi}= BI - \epsilon_{u} g^{-\frac{1}{2}}\odot\Delta_{\mathrm{S},h}(g^{-\frac{1}{2}}\odot)$, the corresponding matrix $-\mathbf{L}_{\phi} = -BI_{N_xN_y}+\epsilon_u G^{-1/2}\Delta_{\mathrm{S},h}G^{-1/2}$ is a Metzler matrix, and all entries of the matrix exponential $e^{-\tau \mathbf{L}_{\phi}}$ are nonnegative for every $\tau>0$.
\end{lemma}

Also, we list some properties of matrix functions useful to the analysis later; one can also see from \cite{du2019maximum,wang2025maximum}.

\begin{lemma}[\cite{higham2008functions}]\label{lemma:maxtix}
    Let $\varphi$ be defined on the spectrum of $A\in \mathbb{C}^{m\times m}$, that is, the values $\varphi(\lambda_{i})$, $1\leq i\leq m$ exist, with $\{\lambda_i\}_{i=1}^m$ being the eigenvalues of $A$, then, 
    \begin{enumerate}
        \item $\varphi(A)$ commutes with $A$, and $\varphi(A^T) = \varphi(A)^T$;
        \item the eigenvalues of $\varphi(A)$ are $\{\varphi(\lambda_i)\}_{i=1}^m$;
        \item $\varphi(P^{-1}AP) = P^{-1}\varphi(A)P$ for any nonsingular matrix $P\in\mathbb{C}^{m\times m}$;
        \item $\frac{d}{ds}e^{sA} = Ae^{sA} = e^{sA}A$ for any $s \in \mathbb{R}$.
    \end{enumerate}
\end{lemma}

By using the differentiation of matrix exponentials in Lemma \ref{lemma:maxtix}, we obtain the exact solution for the ODE system \eqref{eqn:ODE_eqn}
\begin{align*}
    \sqrt{g}\odot u(t+\tau) = e^{-\tau\mathbf{L}_{\phi}}\big(\sqrt{g}\odot u(t)\big) + \int_{0}^{\tau}e^{-(\tau-s)\mathbf{L}_{\phi}}\mathbf{R}_\phi\big(u(t+s)\big)\text{d}s, \quad \forall t\geq0,\ \tau>0.
\end{align*}

\begin{remark}
A key step in our construction is the novel operator splitting method based on dividing both sides of \eqref{eqn:model_OK_fix} by $g^{1/2}$. This reformulation is essential for the subsequent discrete analysis.
\end{remark}

Now we can present the ETD1 and ETDRK2 schemes. Given a time interval $[0,T]$ and an integer $K>0$, we define the time step size $\tau = T/K$ and set $t_n = n\tau$ for $n=0,1,\cdots,K$. Let $U^{n}$ denote the approximate solution at time $t_n$, the ETD1 method is given by
\begin{align}
    \sqrt{g}\odot U^{n+1} = e^{-\mathbf{L}_{\phi}\tau}(\sqrt{g}\odot U^n) + \int_0^{\tau}e^{-\mathbf{L}_{\phi}(\tau-s)}\mathbf{R}_{\phi}(U^n) \text{d}s, \nonumber
\end{align}
which leads to
\begin{align}
    \sqrt{g}\odot U^{n+1} = \phi_0({\mathbf{L}_{\phi}\tau})(\sqrt{g}\odot U^n) + \tau\phi_1({\mathbf{L}_{\phi}\tau}) \mathbf{R}_{\phi}(U^n). \label{eqn:ETD1}
\end{align}
 The ETDRK2 scheme gives
\begin{align}
     \sqrt{g}\odot\Tilde{U}^{n+1} = \ &  e^{-\mathbf{L}_{\phi}\tau}(\sqrt{g}\odot U^n) + \int_0^{\tau}e^{-\mathbf{L}_{\phi}(\tau-s)} \mathbf{R}_{\phi}(U^n) \text{d}s, \nonumber\\
     \sqrt{g}\odot U^{n+1} = \ & e^{-\mathbf{L}_{\phi}\tau}(\sqrt{g}\odot U^n) \nonumber \\
     & + \int_0^{\tau}e^{-\mathbf{L}_{\phi}(\tau-s)}\left[\left(1-\frac{s}{\tau}\right)\mathbf{R}_{\phi}(U^n) + \frac{s}{\tau}\mathbf{R}_{\phi}(\Tilde{U}^{n+1})\right] \text{d}s, \nonumber
\end{align}
 or, equivalently,
\begin{align}\label{eqn:ETDRK2}
    & \sqrt{g}\odot\Tilde{U}^{n+1} = \phi_0(\mathbf{L}_{\phi}\tau)(\sqrt{g}\odot U^n) + \tau\phi_1({\mathbf{L}_{\phi}\tau}) \mathbf{R}_{\phi}(U^n),  \\
    & \sqrt{g}\odot U^{n+1} = \sqrt{g}\odot\Tilde{U}^{n+1} + \tau\phi_2(\mathbf{L}_{\phi}\tau)\left(\mathbf{R}_{\phi}(\Tilde{U}^{n+1}) - \mathbf{R}_{\phi}(U^n)\right). 
\end{align}
Here, $\phi_0(z) := e^{-z}$, $\phi_1(z) := \frac{1-e^{-z}}{z}$, and $\phi_2(z) := \frac{e^{-z}-1+z}{z^2}$ are all positive for all $z>0$.  

\subsubsection{Maximum Bound Principle}

In this subsection, we define the weighted norm for $v\in \mathbb{R}^N$ and the weighted matrix norm for $V\in\mathbb{R}^{N\times N}$ as 
    \begin{align}
        \|v\|_{\infty,w} := \max_{i}\frac{|v_{i}|}{w_i}, \quad \|V\|_{\infty,W} := \max_{i,j}\frac{|V_{ij}|}{W_{ij}}, \label{eqn:weighted_norm}
    \end{align}
    with the corresponding weights $w \in  \mathbb{R}_+^{N}$ and $W\in  \mathbb{R}_+^{N\times N}$. Then the induced weighted matrix norm for $M = (m_{ij})\in\mathbb{R}^{N\times N}$ can be represented as 
    \begin{align}
        \vvvert M _{\infty,w} := \sup_{0 \neq v \in \mathbb{R}^N}\frac{\|Mv\|_{\infty,w}}{\|v\|_{\infty,w}}. \label{eqn:induced_norm}
    \end{align}
    and the corresponding logarithmic norm $\mu_{\infty,w}(M)$ can be defined as
\begin{align}
    \mu_{\infty,w}(M): = \max_{i}\left(m_{ii} + \sum_{j\neq i}|m_{ij}|\frac{w_j}{w_i}\right). \label{eqn:log_norm}
\end{align}

In order to prove the discrete MBP of the ETD1 and ETDRK2 methods, it is essential to analyze the bounds of the exponential-type integrator. Different from the results given in \cite{du2019maximum,wang2025maximum} for the square domain, we prove the new bound with the given weighted norm in the following lemma. 

\begin{lemma}\label{lemma:phi_bdd}
    For any constant $B>0$ and $\tau>0$, we always have 
    \begin{align}
        \vvvert{e^{-\tau\mathbf{L}_{\phi}}}_{\infty,w} \leq e^{-B\tau}, \nonumber
    \end{align}
    where $w = G^{1/2}\mathbb{1}$ and $\mathbb{1}$ is the all-one vector.
\end{lemma}
\begin{proof}
Recall that
\begin{align*}
    \mathbf{L}_{\phi}= BI_{N_xN_y}+\epsilon_u H,
\qquad
H:=-G^{-1/2}\Delta_{\mathrm{S},h}G^{-1/2}.
\end{align*}
We first set $M = -\mathbf{L}_{\phi}$, and compute the logarithmic norm $\mu_{\infty,w}(M)$. Since $w=G^{1/2}\mathbb{1}$ and $\mathbb{1}$ is the all-one vector, we have 
\begin{align*}
    Hw=-G^{-1/2}\Delta_{\mathrm{S},h}G^{-1/2}(G^{1/2}\mathbb{1})= -G^{-1/2}\Delta_{\mathrm{S},h}\mathbb{1} = 0,
\end{align*}
equivalently, for each $i$,
\begin{align*}
    \sum_{j} H_{ij}w_j =0,
\qquad\text{hence}\qquad
\sum_{j} H_{ij}\frac{w_j}{w_i}=0.
\end{align*}
Next, note that the entries of $M$ satisfy
\begin{align*}
m_{ii}=-B-\epsilon_u H_{ii}<0,
\qquad
m_{ij}=-\epsilon_u H_{ij}\geq0 \quad (i\neq j),
\end{align*}
since $H_{ii}>0$ and $H_{ij}\leq0$, then the logarithmic norm can be computed as
\begin{align}
\mu_{\infty,w}(M) &= \max_i\left( m_{ii}+\sum_{j\neq i}|m_{ij}|\frac{w_j}{w_i} \right) = \max_i\left(
-B-\epsilon_u H_{ii} +\sum_{j\neq i}(-\epsilon_u H_{ij})\frac{w_j}{w_i} \right) \nonumber\\
&= \max_i\left( -B-\epsilon_u \sum_j H_{ij}\frac{w_j}{w_i} \right) =\max_i(-B)=-B.
\end{align}
Finally, by the standard estimate relating the matrix norm and the corresponding logarithmic norm \cite{desoer1972measure,soderlind2006logarithmic}, we always have
\begin{align*}
    \vvvert{e^{\tau M}}_{\infty,w} \leq e^{\tau \cdot\mu_{\infty,w}(M)}, \quad \forall \tau\geq 0.
\end{align*}
Therefore, since $M = -\mathbf{L}_{\phi}$, we obtain that
\begin{align*}
    \vvvert{e^{-\tau \mathbf{L}_{\phi}}}_{\infty,w} \leq e^{\tau \cdot\mu_{\infty,w}(-\mathbf{L}_{\phi})} = e^{-B\tau}, \quad \forall\tau\geq0.
\end{align*}
\end{proof}

Denote $L_{W''} = \|W''\|_{C[0,1]}$, $L_{f''} = \|f''\|_{C[0,1]}$, $L_{f'} = \|f'\|_{C[0,1]}$, and $L_{g} = \|g\|_{C[0,1]}$, then we use the following lemma to prove the discrete MBP:

\begin{lemma}\label{lemma:MBPbdd}
    For any function $u\in \mathcal{M}_{h}$ with $0\leq u_{ij} \leq1$ for any $(i,j)\in S_{h}$, if the stabilizer  
    \begin{align}
        B\geq \frac{L_{W''}}{\epsilon_{u}} + \gamma\|(-\Delta_{\mathrm{S},h})^{-1}\|_{L^{\infty},h}L_{f''}L_{g} + M|\Omega|L_{f''}L_g, \nonumber
    \end{align}
    then we have 
    $\mathbf{R}_{\phi}(u)\geq 0$ and $\|\mathbf{R}_{\phi}(u)\|_{\infty,\sqrt{g}}\leq B$.
\end{lemma}

\begin{proof}
    Given any function $v,w\in \mathcal{M}_{h}$ with $v_{ij},w_{ij} \in [0,1]$, we can compute 
    \begin{align}
        \mathbf{R}_{\phi}(v) - \mathbf{R}_{\phi}(w) 
        = \ & B\sqrt{g}\odot(v-w) - \frac{1}{\epsilon_{u}} \sqrt{g}\odot W''(\zeta)(v-w) \nonumber\\
     &-\gamma\sqrt{g}\odot
    \biggl[
    \begin{aligned}[t]
        &(-\Delta_{\mathrm{S},h})^{-1}
        \bigl[g\odot f(v)\bigr]
        \odot
        \bigl(f'(w)-f''(\xi)(v-w)\bigr)
        \\
        &\quad
        -(-\Delta_{\mathrm{S},h})^{-1}
        \bigl[g\odot f(w)\bigr]\odot f'(w)
    \end{aligned}
    \biggr]
    \nonumber\\
    &-M\sqrt{g}\odot
    \biggl[
    \begin{aligned}[t]
        &\left\langle g\odot(f(v)-\bar{u}),\mathbb{1}\right\rangle_h
        \bigl(f'(w)-f''(\xi)(v-w)\bigr)
        \\
        &\quad
        -\left\langle g\odot(f(w)-\bar{u}),\mathbb{1}\right\rangle_h f'(w)
    \end{aligned}
    \biggr]. \nonumber
    \end{align}
    Here, $\zeta,\xi$ are located between $v$ and $w$. Next, taking $u\in \mathcal{M}_{h}$ with $0\leq u_{ij} \leq1$ for any $(i,j)\in S_{h}$, if we choose $v = u$ and $w = 0$, then we have $f'(0) = 0$, $0\leq f(u)\leq1$, $\|f(u) - \bar{u}\|_{L^{\infty},h} = \max\{1-\bar{u},\bar{u}\}\leq1$, and 
    \begin{align}
        & \mathbf{R}_{\phi}(u) - \mathbf{R}_{\phi}(0) \nonumber\\ = \ & B\sqrt{g}\odot u - \frac{1}{\epsilon_{u}} W''(\zeta)\sqrt{g}\odot u \nonumber\\
        & - \gamma \sqrt{g}\odot(-\Delta_{\mathrm{S},h})^{-1}\left[g\odot f(u)\right]\odot f''(\xi)u - M\sqrt{g}\odot\left\langle g\odot(f(u)-\bar{u}),\mathbb{1}\right\rangle_h\odot f''(\xi)u  \nonumber \\
        \geq\ & \left(B -\frac{L_{W''}}{\epsilon_{u}} - \gamma\|(-\Delta_{\mathrm{S},h})^{-1}\|_{L^{\infty},h}L_{f''}L_{g} - M|\Omega|L_{f''}L_g\right)\sqrt{g}\odot u \geq 0. \nonumber
    \end{align}
    Therefore, we have $\mathbf{R}_{\phi}(u) \geq \mathbf{R}_{\phi}(0) = 0$.

    Similarly, we take $v = u$ and $w = 1$, then $f'(1) = 0$, $0\leq1-u\leq 1$, and 
        \begin{align}
        & \mathbf{R}_{\phi}(u) - \mathbf{R}_{\phi}(1) \nonumber\\ 
        = \ & -B\sqrt{g}\odot(1-u) + \frac{1}{\epsilon_{u}} \sqrt{g}\odot W''(\zeta)(1-u) \nonumber\\
        & + \gamma \sqrt{g}\odot(-\Delta_{\mathrm{S},h})^{-1}\left[g\odot f(u)\right]\odot f''(\xi)(1-u) \nonumber\\
        & + M\sqrt{g}\odot\left\langle g\odot(f(u)-\bar{u}),1\right\rangle_h\odot f''(\xi)(1-u)  \nonumber \\
        \leq\ & -\left(B -\frac{L_{W''}}{\epsilon_{u}} - \gamma\|(-\Delta_{\mathrm{S},h})^{-1}\|_{L^{\infty},h}L_{f''}L_{g} - M|\Omega|L_{f''}L_g\right)\sqrt{g}\odot(1-u) \leq 0, \nonumber
    \end{align}
    that yields $\|\mathbf{R}_{\phi}(u)\|_{\infty,\sqrt{g}} \leq \|\mathbf{R}_{\phi}(1)\|_{\infty,\sqrt{g}} = \|B\sqrt{g}\|_{\infty,\sqrt{g}} = B$.
\end{proof}

Next, we prove the discrete MBP for the ETD1 and ETDRK2 schemes:
\begin{theorem}\label{thm:MBP}
    Assume that the initial condition $U^0$ satisfies $0 \leq U^0 \leq 1$, then for any time step size $\tau > 0$, the ETD1 and ETDRK2 schemes both preserve the discrete maximum bound principle (MBP), that is 
    \begin{align}\label{con:MPP}
        0 \leq U^n\leq 1,  \quad \forall \ n\geq0. \nonumber
    \end{align}
    provided that 
    \begin{align}
        B\geq \frac{L_{W''}}{\epsilon_{u}} + \gamma\|(-\Delta_{\mathrm{S},h})^{-1}\|_{L^{\infty},h}L_{f''}L_{g} + M|\Omega|L_{f''}L_g. \label{eqn:cond_MBP}
    \end{align}
\end{theorem}

\begin{proof}
     At first, we prove that the ETD1 scheme preserves the MBP by induction. Following the similar idea as in \cite{wang2025maximum} with a special norm, assume that at step $n=k$, the numerical solution satisfies $0\leq U^k \leq 1$, then $\|\sqrt{g}\odot U^k\|_{\infty,\sqrt{g}}\leq 1$ and $\sqrt{g}\odot U^k \geq 0$. From Lemma \ref{lemma:MBPbdd}, we have $\mathbf{R}_{\phi}(U^k) \geq 0$. Then, at the next step $n = k+1$, the ETD1 scheme gives
    \begin{align}
        \sqrt{g}\odot U^{k+1} = e^{-\mathbf{L}_{\phi}\tau}\left(\sqrt{g}\odot U^k\right)  + \int_0^{\tau}e^{-\mathbf{L}_{\phi}(\tau-s)} \text{d}s \ \mathbf{R}_{\phi}(U^k) \geq 0. \nonumber 
    \end{align} 
 Thus, by Lemma \ref{lemma:LuPD}, we obtain that $e^{-\mathbf{L}_{\phi}\tau}$ and $e^{-\mathbf{L}_{\phi}(\tau-s)}$ are both nonnegative. Then since $g > 0$, we can get $U^{k+1}\geq 0 $. Next, taking the weighted norm yields
    \begin{align}
        \|\sqrt{g}\odot U^{k+1}\|_{\infty,\sqrt{g}} \leq \ & \vvvert{e^{-\mathbf{L}_{\phi}\tau}}_{\infty,\sqrt{G}\mathbb{1}}\|\sqrt{g}\odot U^{k}\|_{\infty,\sqrt{g}} \nonumber \\
        &+ \int_0^{\tau}\vvvert{e^{-\mathbf{L}_{\phi}(\tau-s)}}_{\infty,\sqrt{G}\mathbb{1}} \text{d}s\ \| \mathbf{R}_{\phi}(U^k) \|_{\infty,\sqrt{g}}. \nonumber
    \end{align}
    By Lemma \ref{lemma:phi_bdd}, we have the estimates
    \begin{align}
        \vvvert{e^{-\mathbf{L}_{\phi}\tau}}_{\infty,\sqrt{G}\mathbb{1}} \leq e^{-B\tau}, \quad \int_0^{\tau}\vvvert{e^{-\mathbf{L}_{\phi}(\tau-s)}}_{\infty,\sqrt{G}\mathbb{1}} \text{d}s \leq \frac{1-e^{-B\tau}}{B}. \nonumber
    \end{align}
    Moreover, Lemma \ref{lemma:MBPbdd} leads to $\|\mathbf{R}_{\phi}(U^{k})\|_{\infty,\sqrt{g}}\leq B$. Therefore, combining these estimates, we obtain
    \begin{align}
        \|\sqrt{g}\odot U^{k+1}\|_{\infty,\sqrt{g}} \leq e^{-B\tau} + \frac{1-e^{-B\tau}}{B}B \leq 1. \nonumber
    \end{align}
    By the definition of weighted norm \eqref{eqn:weighted_norm}, the inequality leads to $ |U^{k+1}_{ij}|\leq 1$ with $U^{k+1}_{ij}$ any entries in $U^{k+1}$. Therefore, we obtain $U^{k+1}\leq 1$. Thus, the ETD1 scheme preserves the discrete MBP.

    Then, we apply a similar argument to the ETDRK2 scheme. Again, assume $0 \leq U^k \leq 1$. From the MBP of the ETD1 scheme, the intermediate value $\Tilde{U}^{k+1}$ defined there also satisfies $0\leq \Tilde{U}^{k+1} \leq 1$. Then, for $n = k+1$, and $s\in[0,\tau]$, one has 
    \begin{align}
        \sqrt{g}\odot U^{k+1} = e^{-\mathbf{L}_{\phi}\tau}(\sqrt{g}\odot U^n)  + \int_0^{\tau}e^{-\mathbf{L}_{\phi}(\tau-s)} \left[\left(1-\frac{s}{\tau}\right)\mathbf{R}_{\phi}(U^{k}) + \frac{s}{\tau}  \mathbf{R}_{\phi}(\Tilde{U}^{k+1})\right] \text{d}s, \nonumber
    \end{align}
    that results in $U^{k+1}\geq0$ from Lemma \ref{lemma:LuPD} and the ETD1 scheme. Next, Lemma \ref{lemma:MBPbdd} also leads to
    \begin{align}
         \left\|\left(1-\frac{s}{\tau}\right)\mathbf{R}_{\phi}(U^{k}) + \frac{s}{\tau}  \mathbf{R}_{\phi}(\Tilde{U}^{k+1})\right\|_{\infty,\sqrt{g}} 
        \leq  \ \left(1-\frac{s}{\tau}\right)B + \frac{s}{\tau}B = B. \nonumber
    \end{align}
    Therefore, the inequality follows 
    \begin{align}
        & \|\sqrt{g}\odot U^{k+1}\|_{\infty,\sqrt{g}} \nonumber \\
        \leq \ &   \vvvert{e^{-\mathbf{L}_{\phi}\tau}}_{\infty,\sqrt{G}\mathbb{1}}\|\sqrt{g}\odot U^{k}\|_{\infty,\sqrt{g}} \nonumber \\
        & + \int_0^{\tau} \vvvert{e^{-\mathbf{L}_{u}(\tau-s)}}_{\infty,\sqrt{G}\mathbb{1}} \text{d}s\left\|\left(1-\frac{s}{\tau}\right)\mathbf{R}_{\phi}(U^{k}) + \frac{s}{\tau}  \mathbf{R}_{\phi}(\Tilde{U}^{k+1})\right\|_{\infty,\sqrt{g}} \nonumber \\
        \leq \ &  e^{-B\tau} + \frac{1-e^{-B\tau}}{B}B \leq 1. \nonumber
    \end{align}
    Hence, $U^{k+1}\leq 1$, and the ETDRK2 scheme also preserves the discrete MBP.
\end{proof}

\subsubsection{Energy Stability}

We then establish energy decay for the discrete \\
membrane-associated OK energy  
\begin{align}
    E_{\phi,h}[U^n] =\ & \frac{\epsilon_u}{2}\left\langle -\Delta_{\mathrm{S},h}U^n,U^n \right\rangle_{h} + \frac{1}{\epsilon_u}\left\langle g \odot W(U^n), \mathbb{1}\right\rangle_{h} \nonumber\\
    & + \frac{\gamma}{2}\left\langle (-\Delta_{\mathrm{S},h})^{-1}\left(g\odot f(U^{n})\right) , g\odot f(U^{n})\right\rangle_{h} + \frac{M}{2}\left\langle g\odot(f(U^{n}) - \bar{u}),\mathbb{1} \right\rangle_{h}^2, \label{eqn:Energy_dis}
\end{align}
for the ETD1 and ETDRK2 schemes in the following theorems. 

\begin{theorem}\label{thm:Energy}
    Given  
    \begin{align}
        B = \frac{L_{W''}}{\epsilon_{u}}+\left(\frac{\gamma}{2}\|(-\Delta_{\mathrm{S},h})^{-1}\|_{L^{\infty},h}+\frac{M}{2}|\Omega|\right)(L_{f''}+L_{f'}^2)L_g, \label{eqn:cond_ES}
    \end{align}
    and any time step $\tau>0$, we always have the energy stability for both the ETD1 \eqref{eqn:ETD1} and ETDRK2 \eqref{eqn:ETDRK2} schemes, i.e. 
    \begin{align*}
    E_{\phi,h}[U^{n+1}] \leq E_{\phi,h}[U^n].
    \end{align*}
\end{theorem}

\begin{proof}
We first prove for the ETD1 scheme \eqref{eqn:ETD1}. The energy difference between two successive time steps $U^{n+1}$ and $U^n$ is given by
\begin{align}
    &E_{\phi,h}[U^{n+1}]-E_{\phi,h}[U^n]
    ={}
    \underbrace{
    \begin{aligned}[t]
        &\frac{\epsilon_u}{2}
        \left\langle
        -\Delta_{\mathrm{S},h}U^{n+1},U^{n+1}
        \right\rangle_h
        \\
        &\quad
        -\frac{\epsilon_u}{2}
        \left\langle
        -\Delta_{\mathrm{S},h}U^{n},U^{n}
        \right\rangle_h
    \end{aligned}
    }_{\textbf{I}}
    +
    \underbrace{
    \begin{aligned}[t]
        &\frac{1}{\epsilon_u}
        \left\langle
        g\odot W(U^{n+1}),\mathbb{1}
        \right\rangle_h
        \\
        &\quad
        -\frac{1}{\epsilon_u}
        \left\langle
        g\odot W(U^{n}),\mathbb{1}
        \right\rangle_h
    \end{aligned}
    }_{\textbf{II}}
    \nonumber\\
    &+
    \underbrace{
    \begin{aligned}[t]
        &\frac{\gamma}{2}
        \left\langle
        (-\Delta_{\mathrm{S},h})^{-1}
        \bigl(g\odot f(U^{n+1})\bigr),
        g\odot f(U^{n+1})
        \right\rangle_h
        \\
        &\quad
        -\frac{\gamma}{2}
        \left\langle
        (-\Delta_{\mathrm{S},h})^{-1}
        \bigl(g\odot f(U^{n})\bigr),
        g\odot f(U^{n})
        \right\rangle_h
    \end{aligned}
    }_{\textbf{III}}
    +
    \underbrace{
    \begin{aligned}[t]
        &\frac{M}{2}
        \left\langle
        g\odot\bigl(f(U^{n+1})-\bar{u}\bigr),\mathbb{1}
        \right\rangle_h^2
        \\
        &\quad
        -\frac{M}{2}
        \left\langle
        g\odot\bigl(f(U^{n})-\bar{u}\bigr),\mathbb{1}
        \right\rangle_h^2
    \end{aligned}
    }_{\textbf{IV}} .
    \nonumber
\end{align}
     Using the identity
    \begin{align}
         a\cdot(a-b) = \frac{1}{2}\left[a^2-b^2+(a-b)^2\right], \quad\quad b\cdot(a-b) = \frac{1}{2}\left[a^2-b^2-(a-b)^2\right], \nonumber
    \end{align}
 the term $\mathbf{I}$ satisfies 
    \begin{align}
        \textbf{I} = \ & \epsilon_u\left\langle -\Delta_{\mathrm{S},h}U^{n+1},U^{n+1}-U^{n} \right\rangle_{h} - \frac{\epsilon_u}{2}\left\langle -\Delta_{\mathrm{S},h}(U^{n+1}-U^{n}),U^{n+1}-U^{n} \right\rangle_{h}. \nonumber 
    \end{align}
    By Taylor expansion, $W'(U^{n})(U^{n+1}-U^n) = W(U^{n+1})-W(U^n)-\frac{W''(\xi^n)}{2}(U^{n+1}-U^n)^2$, where $\xi^n$ is between $U^n$ and $U^{n+1}$, term $\mathbf{II}$ is bounded by
     \begin{align}
        \textbf{II} 
        \leq \frac{1}{\epsilon_u}\left\langle \sqrt{g}\odot W'(U^{n}),\sqrt{g}\odot (U^{n+1}-U^{n}) \right\rangle_{h} + \frac{L_{W''}}{\epsilon_u}\| \sqrt{g}\odot (U^{n+1}-U^{n}) \|_{L^2,h}^2. \nonumber
    \end{align}
 Similarly, term $\mathbf{III}$ and $\mathbf{IV}$ are given by
 \begin{align}
        \textbf{III} 
        \leq \ & \gamma\left\langle \sqrt{g}\odot (-\Delta_{\mathrm{S},h})^{-1}\left(g\odot f(U^{n})\right)\odot f'(U^n), \sqrt{g}\odot (U^{n+1} - U^{n}) \right\rangle_{h} \nonumber\\
        & + \frac{\gamma}{2}(L_{f''}+L_{f'}^2)L_g\|(-\Delta_{\mathrm{S},h})^{-1}\|_{L^{\infty},h}\| \sqrt{g}\odot (U^{n+1}-U^{n}) \|_{L^2,h}^2, \nonumber\\
        \textbf{IV}
        \leq \ & M\left\langle \sqrt{g}\odot \left\langle g\odot (f(U^{n}) - \bar{u}),\mathbb{1} \right\rangle_{h}\odot f'(U^{n}),\sqrt{g}\odot (U^{n+1} - U^{n}) \right\rangle_{h} \nonumber \\
        & + \frac{M}{2}|\Omega|(L_{f''}+L_{f'}^2)L_g\| \sqrt{g}\odot (U^{n+1}-U^{n}) \|_{L^2,h}^2 .\nonumber
    \end{align}
    Combining the above inequalities and taking 
    \begin{align*}
        B = \frac{L_{W''}}{\epsilon_u}+\left(\frac{\gamma}{2}\|(-\Delta_{\mathrm{S},h})^{-1}\|_{L^{\infty},h}+\frac{M}{2}|\Omega|\right)(L_{f''}+L_{f'}^2)L_g,
    \end{align*}
    we can get the difference between energies as 
    \begin{align}
         E_{\phi,h}[U^{n+1}] - E_{\phi,h}[U^n] \leq \ & \left\langle \mathbf{L}_{\phi}(\sqrt{g}\odot U^{n+1}) - \mathbf{R}_{\phi}(U^n), \sqrt{g}\odot(U^{n+1} - U^{n})\right\rangle_{h}. \nonumber
    \end{align}
    For the ETD1 scheme, we can solve for $\mathbf{R}_{\phi}(U^n)$ as
        \begin{align}
        \mathbf{R}_{\phi}(U^n) = \ & \left(I-e^{-\mathbf{L}_{\phi}\tau}\right)^{-1}\mathbf{L}_{\phi}\left(\sqrt{g}\odot(U^{n+1} -U^n)\right)+\mathbf{L}_{\phi}(\sqrt{g}\odot U^{n}). \nonumber
    \end{align}
    That leads to 
    \begin{align}
        \mathbf{L}_{\phi}(\sqrt{g}\odot U^{n+1}) - \mathbf{R}_{\phi}(U^n) = \left[\mathbf{L}_{\phi} - \left(I-e^{-\mathbf{L}_{\phi}\tau}\right)^{-1}\mathbf{L}_{\phi}\right]\left(\sqrt{g}\odot(U^{n+1} -  U^n)\right). \nonumber
    \end{align}
     Define a function $\eta_1(a)$ as $\eta_1(a): = a-\frac{a}{1-e^{-a}}$, $a\neq 0$, 
     and let $H_1 = \frac{\eta_1(\mathbf{L}_{\phi}\tau)}{\tau} = \mathbf{L}_{\phi} - \left(I-e^{-\mathbf{L}_{\phi}\tau}\right)^{-1}\mathbf{L}_{\phi}$. For any $a>0$, $\eta_1(a)<0$. Since $\mathbf{L}_{\phi}$ is symmetric and positive definite, $H_1$ is symmetric and negative definite. The energy stability for the ETD1 scheme is given as
     \begin{equation}
         E_{\phi,h}[U^{n+1}] - E_{\phi,h}[U^n]\leq \left\langle H_1(\sqrt{g}\odot(U^{n+1} - U^{n})), \sqrt{g}\odot(U^{n+1} - U^{n})\right\rangle_{h} \leq 0. \label{eqn:Energydiff_ETD1}
     \end{equation}

     Next, we move to the ETDRK2 scheme \eqref{eqn:ETDRK2} by applying the result for the ETD1 scheme. Using the same approach for the ETD1 scheme, we can find the difference in the energies between two successive time steps as 
    \begin{align}
        E_{\phi,h}[U^{n+1}] - E_{\phi,h}[U^{n}] 
         = \ & E_{\phi,h}[U^{n+1}] - E_{\phi,h}[\Tilde{U}^{n+1}] + E_{\phi,h}[\Tilde{U}^{n+1}] - E_{\phi,h}[U^{n}] \nonumber\\
         \leq \ & \underbrace{\left\langle \mathbf{L}_{\phi}(\sqrt{g}\odot U^{n+1}) - \mathbf{R}_{\phi}(\Tilde{U}^{n+1}), \sqrt{g}\odot (U^{n+1} - \Tilde{U}^{n+1})\right\rangle_{h}}_{\mathbf{V}} \nonumber\\
         & + \underbrace{\left\langle \mathbf{L}_{\phi}(\sqrt{g}\odot \Tilde{U}^{n+1}) - \mathbf{R}_{\phi}(U^n), \sqrt{g}\odot (\Tilde{U}^{n+1} - U^{n})\right\rangle_{h}}_{\mathbf{VI}}. \nonumber
    \end{align}
    Then, for the ETDRK2 scheme \eqref{eqn:ETDRK2}, we have 
    \begin{align}
        \sqrt{g}\odot U^{n+1} = \sqrt{g}\odot \Tilde{U}^{n+1}+\tau\frac{e^{-\mathbf{L}_{\phi}\tau}-1+\mathbf{L}_{\phi}\tau}{(\mathbf{L}_{\phi}\tau)^2}\left(\mathbf{R}_{\phi}(\Tilde{U}^{n+1})-\mathbf{R}_{\phi}(U^n)\right). \nonumber
    \end{align}
    It can be reformulated by
    \begin{align}
        \mathbf{R}_{\phi}(\Tilde{U}^{n+1}) = \mathbf{R}_{\phi}(U^n) + \tau \mathbf{L}_{\phi}^2\left(e^{-\mathbf{L}_{\phi}\tau}-1+\mathbf{L}_{\phi}\tau\right)^{-1}\left(\sqrt{g}\odot (U^{n+1} - \Tilde{U}^{n+1})\right). \nonumber
    \end{align}
    Then, $\mathbf{V}$ and $\mathbf{VI}$ can be rewritten by 
    \begin{align}
    \mathbf{V}
    ={}&
    \left\langle
    \begin{aligned}
        &\left[
        \mathbf{L}_{\phi}
        -
        \left(I-e^{-\mathbf{L}_{\phi}\tau}\right)^{-1}
        \mathbf{L}_{\phi}
        \right]
        \left(\sqrt{g}\odot(\widetilde{U}^{n+1}-U^n)\right),
        \\
        &\sqrt{g}\odot(U^{n+1}-\widetilde{U}^{n+1})
    \end{aligned}
    \right\rangle_h
    \nonumber\\
    &+
    \left\langle
    \begin{aligned}
        &\left[
        \mathbf{L}_{\phi}
        -
        \tau\mathbf{L}_{\phi}^2
        \left(
        e^{-\mathbf{L}_{\phi}\tau}
        -I
        +\mathbf{L}_{\phi}\tau
        \right)^{-1}
        \right]
        \left(\sqrt{g}\odot(U^{n+1}-\widetilde{U}^{n+1})\right),
        \\
        &\sqrt{g}\odot(U^{n+1}-\widetilde{U}^{n+1})
    \end{aligned}
    \right\rangle_h ,
    \nonumber\\[1mm]
    \mathbf{VI}
    ={}&
    \left\langle
    \begin{aligned}
        &\left[
        \mathbf{L}_{\phi}
        -
        \left(I-e^{-\mathbf{L}_{\phi}\tau}\right)^{-1}
        \mathbf{L}_{\phi}
        \right]
        \left(\sqrt{g}\odot(\widetilde{U}^{n+1}-U^n)\right),
        \\
        &\sqrt{g}\odot(\widetilde{U}^{n+1}-U^n)
    \end{aligned}
    \right\rangle_h .
    \nonumber
\end{align}
    Denote $H_2 = \mathbf{L}_{\phi} - \tau \mathbf{L}_{\phi}^2\left(e^{-\mathbf{L}_{\phi}\tau}-1+\mathbf{L}_{\phi}\tau\right)^{-1}$, then we have 
        \begin{align}
        & E_{\phi,h}[U^{n+1}] - E_{\phi,h}[U^{n}] \nonumber \\
         \leq \ & \frac{1}{2} \left\langle \sqrt{g}\odot(\Tilde{U}^{n+1} - U^{n}), H_1\left(\sqrt{g}\odot(\Tilde{U}^{n+1} -U^n)\right)\right\rangle_{h} \nonumber\\
        & + \frac{1}{2}\left\langle \sqrt{g}\odot(U^{n+1} - U^{n}), H_1\left(\sqrt{g}\odot(U^{n+1} -U^n)\right)\right\rangle_{h} \nonumber\\
        & + \left\langle \sqrt{g}\odot(U^{n+1} - \Tilde{U}^{n+1}), \left(H_2-\frac{1}{2}H_1\right)\left(\sqrt{g}\odot(U^{n+1} - \Tilde{U}^{n+1})\right)\right\rangle_{h}. \label{eqn:Energydiff_ETDRK2}
    \end{align}
    Define a function $\eta_2(a)$ as $\eta_2(a) := a - \frac{a^2}{e^{-a}-1+a}$, $a\neq0$,
    then $H_1 = \frac{\eta_1(\mathbf{L}_{\phi}\tau)}{\tau}$ and $H_2 = \frac{\eta_2(\mathbf{L}_{\phi}\tau)}{\tau}$. Since $H_1$ is symmetric and negative definite, we have 
    \begin{align}
        \left\langle \sqrt{g}\odot(\Tilde{U}^{n+1} - U^{n}), H_1\left(\sqrt{g}\odot(\Tilde{U}^{n+1} -U^n)\right)\right\rangle_{h} \leq 0, \nonumber\\
        \left\langle \sqrt{g}\odot(U^{n+1} - U^{n}), H_1\left(\sqrt{g}\odot(U^{n+1} -U^n)\right)\right\rangle_{h} \leq0. \nonumber
    \end{align}
    Define a new function $\eta(a)$ as
    \begin{align}
          \eta(a) = \eta_2(a) - \frac{1}{2}\eta_{1}(a) = \frac{a(e^{-a}(3+a-e^{-a})-2)}{2(e^{-a}-1+a)(1-e^{-a})}. \nonumber
    \end{align}
    For any $a>0$, we always have $\eta(a)\leq 0$. Thus $H_2 - \frac{1}{2}H_1$ is symmetric and negative definite since $\mathbf{L}_{\phi}$ is symmetric and positive definite. Thus, we have
    \begin{align}
        & \left\langle \sqrt{g}\odot(U^{n+1} - \Tilde{U}^{n+1}), \left(H_2-\frac{1}{2}H_1\right)\left(\sqrt{g}\odot(U^{n+1} - \Tilde{U}^{n+1})\right)\right\rangle_{h} \leq 0. \nonumber
    \end{align}
    Therefore, we conclude that 
    \begin{align}
        E_{\phi,h}[U^{n+1}] - E_{\phi,h}[U^{n}] \leq 0, \nonumber
    \end{align}
    which establishes the unconditional energy stability for the ETDRK2 scheme.
\end{proof}

\begin{remark}
The proof of Theorem \ref{thm:Energy} follows the general strategy used for ETD schemes in earlier work \cite{du2019maximum,wang2025maximum,fu2022energy}, but requires a special treatment of the membrane geometry encoded by the localization function $g$. This treatment is based on the operator splitting method given in \eqref{eqn:OK_FD_operatorsplit}. Using this method, the discrete energy difference can be written in the inner product that appears in \eqref{eqn:Energydiff_ETD1}, \eqref{eqn:Energydiff_ETDRK2}, where the key term shares the common factor $\sqrt{g}\odot(U^{n+1}-U^n)$. In this setting, the negative definiteness of matrices $H_1$ and $H_2 - \frac{1}{2}H_1$ leads to the negativity of the energy difference, which yields discrete energy decay for both ETD1 and ETDRK2 schemes. 
\end{remark}

\section{Efficient ETD Schemes for the Coupled Phase-field Model}\label{Sec:ETDforfull}

Section~\ref{Sec:ETDforMOK} developed the ETD framework for the membrane-associated OK model on a fixed phase-field geometry and established its main analytical properties, including the discrete MBP Theorem \ref{thm:MBP} and energy stability Theorem \ref{thm:Energy}. These results clarify how the localization function $g(\phi)$ affects the operator structure and how it can be incorporated into a stable ETD formulation, and show that the ETD methods are particularly well suited for studying protein segregation on diffuse membranes. Since the membrane-associated OK dynamics are a key component of the phase-field model of multicomponent membranes \eqref{eqn:model_force}–\eqref{eqn:model_OK}, this ETD framework provides both the mathematical basis and the numerical foundation for the coupled system.

We now turn to the coupled system in this section, where the phase-field cell $\phi$ and protein distribution $u$ evolve simultaneously. In this setting, the membrane $g(\phi)$ is no longer fixed, so the operator splitting method \eqref{eqn:model_OK_half} used in Section~\ref{Sec:ETDforMOK} becomes less suitable, both in formulation and in implementation. In particular, the exponential integrator for the membrane-associated OK equation now involves a linear operator that depends explicitly on $g(\phi)$, and therefore must be updated at every time step; moreover, this operator cannot be treated by the FFT-based fast algorithms used in \cite{du2019maximum,wang2025maximum,ju2015fast}, which makes it the dominant computational cost. For this reason, we construct new ETD schemes that are efficient within an FFT framework by adopting a spectral spatial discretization, together with the operator-splitting stabilization in \cite{wang2016efficient} for the force-balance phase-field equation \eqref{eqn:model_force} and a separate stabilization for the membrane-associated OK dynamics \eqref{eqn:model_OK}. Because of the complexity of the coupled system, a rigorous numerical analysis for these FFT-based ETD schemes remains open. Instead, this section focuses on developing efficient and reliable algorithms for long-time simulations of the full multicomponent membrane model \eqref{eqn:model_force}–\eqref{eqn:model_OK} in three dimensions.

\subsection{Linear Splittings for the Coupled System}
We adopt a similar linear splitting scheme in stabilized numerical methods for the phase field elastic bending energy model as in \cite{wang2016efficient} and rewrite the equation \eqref{eqn:model_force} with stabilization constants $A_1, A_2>0$ as 
\begin{align}
    \frac{\partial\phi}{\partial t} = \ & \left[-\frac{\kappa}{\mu}\left(\Delta^2\phi - \frac{A_1}{\epsilon_{\phi}^2}\Delta\phi\right)+\left(\frac{\lambda_{\mathrm{surf}}}{\mu} + \frac{\kappa A_2}{\mu\epsilon_{\phi}^2}\right)\left(\Delta\phi-\frac{A_1}{\epsilon_{\phi}^2}\phi\right)\right] \nonumber\\
    & + \left[ \begin{aligned} &\frac{\kappa}{\mu\epsilon_{\phi}^2} \left[ \Delta\bigl(W'(\phi)-A_1\phi\bigr) + \bigl(W''(\phi)-A_2\bigr) \left( \Delta\phi - \frac{A_1}{\epsilon_{\phi}^2}\phi \right) \right] \\ &\quad - \left( \frac{\kappa}{\mu\epsilon_{\phi}^4}W''(\phi) + \frac{\lambda_{\mathrm{surf}}}{\mu\epsilon_{\phi}^2} \right) \bigl(W'(\phi)-A_1\phi\bigr) \end{aligned} \right] \nonumber\\
    &  + \frac{\lambda_{\mathrm{line}}}{\mu}\left(\epsilon_{u} \Delta_{\mathrm{S}}u - \frac{1}{\epsilon_{u}}g(\phi)W'(u)\right)|\nabla\phi| - \frac{M_{\mathrm{area}}}{\mu}\left(\int_{\Omega}\phi \ dx - A_0\right)|\nabla\phi| \nonumber \\ 
    &  + \frac{\alpha}{\mu} (u + u_0)|\nabla\phi|\left(\epsilon_{\phi}\Delta\phi-\frac{1}{\epsilon_{\phi}}W'(\phi)\right) = -\mathcal{L}_{\phi}(\phi) + \mathcal{R}_{1}(\phi,u). \nonumber
\end{align}
Here, we take the linear term 
\begin{align}
    \mathcal{L}_{\phi}(\phi) = \left(\frac{\kappa}{\mu}\Delta^2 - \frac{\kappa (A_1 + A_2) + \lambda_{\mathrm{surf}}\epsilon_{\phi}^2}{\mu\epsilon_{\phi}^2}\Delta + \frac{\kappa A_1A_2 + \lambda_{\mathrm{surf}}A_1\epsilon_{\phi}^2}{\mu\epsilon_{\phi}^4}\right)\phi, \label{eqn:linear_force}
\end{align}
and the all the rest (nonlinear) terms $\mathcal{R}_{1}(\phi,u)$.
For the equation \eqref{eqn:model_OK}, to apply the spectral collocation method for the linear term, the stabilized form becomes
\begin{align}
         \frac{\partial(gu)}{\partial t} = \ & \left[\epsilon_{u} B_1 \Delta(gu)- B_2gu\right] \nonumber\\
         & + \epsilon_{u}\left( \Delta_{\mathrm{S}}u - B_1\Delta(gu) \right) + B_2gu - \frac{1}{\epsilon_{u}} gW'(u) + \gamma g\Delta_{\mathrm{S}}^{-1}\Big(g(\phi)(u-\bar{u})\Big) \nonumber \\  & - M g \left(\int_{\Omega}g(\phi)\left(u-\bar{u}\right)\ \text{d}x \right) - \nabla\cdot(gu\mathbf{v}) = -\mathcal{L}_{u}(gu) + \mathcal{R}_{2}(\phi,u), \nonumber
\end{align}
where $B_1, B_2>0$ are stabilization constants. Here, the linear term is
\begin{align}
   \mathcal{L}_{u}(gu) =  \left( B_2 - \epsilon_{u} B_1 \Delta \right)gu, \label{eqn:linear_OK}
\end{align}
and the rest (nonlinear) terms are $\mathcal{R}_{2}(\phi,u)$.
Therefore, the system \eqref{eqn:model_force}-\eqref{eqn:model_OK} can be reformulated as
\begin{align}
     & \frac{\partial\phi}{\partial t} + \mathcal{L}_{\phi}(\phi) =   \mathcal{R}_{1}(\phi,u), \label{eqn:force_spec}\\
    & \frac{\partial (gu)}{\partial t} + \mathcal{L}_{u}(gu) =  \mathcal{R}_{2}(\phi,u). \label{eqn:OK_spec}
\end{align}

\begin{remark}
The linear splitting \eqref{eqn:linear_force} for the force-balance phase-field equation \eqref{eqn:force_spec} extracts the dominant fourth- and second-order terms into the linear operator, which improves the temporal stability of the ETD scheme. For the membrane-associated OK dynamics \eqref{eqn:OK_spec}, the splitting \eqref{eqn:linear_OK} is designed so that the linear operator involves only the standard Laplacian and stabilization constants, and therefore is independent of the evolving phase-field cell $\phi$. Consequently, the geometry-dependent operators in \eqref{eqn:model_OK} are confined to the nonlinear part. With these choices, the linear parts of \eqref{eqn:force_spec}–\eqref{eqn:OK_spec} can be treated efficiently by FFT within the ETD framework. In this way, both splittings produce linear operators that are simple, stable, and well-suited for fast spectral implementation in long-time simulations.
\end{remark}

\subsection{Spectral Spatial Discretization}

We discretize the spatial linear operators by using the spectral collocation approximation. Consider a similar domain $\Omega$ in three dimensions as in Section \ref{Sec:ETDforMOK} with the uniform mesh $N_x$, $N_y$, $N_z$ all positive even integers and mesh sizes $h_x = \frac{2X}{N_x}$, $h_y = \frac{2Y}{N_y}$, $h_z = \frac{2Z}{N_z}$.
Denote the sets of approximate solutions $\Phi = (\phi_{ijk})_{0:N_x-1,0:N_y-1,0:N_z-1}$, $\mathbf{GU} = (g_{ijk}u_{ijk})_{0:N_x-1,0:N_y-1,0:N_z-1}$, and the Laplacian operator in the spectral space is corresponding to the following spectrum
\begin{align}
    \lambda_{ijk} = -\lambda_{x}^2(i) -\lambda_{y}^2(j) - \lambda_{z}^2(k), \nonumber
\end{align}
where 
\begin{align}
    \lambda_{x}(i) = \left\{
    \begin{array}{ll}
      \frac{\pi}{X}i,   & 0\leq i\leq N_x/2, \\
       \frac{\pi}{X}(N_x-i),  &  N_x/2+1\leq i\leq N_x-1,
    \end{array}\right. \nonumber\\
        \lambda_{y}(j) = \left\{
    \begin{array}{ll}
      \frac{\pi}{Y}j,   & 0\leq j\leq N_y/2, \\
       \frac{\pi}{Y}(N_y-j),  &  N_y/2+1\leq j\leq N_y-1.
    \end{array}\right. \nonumber \\
    \lambda_{z}(k) = \left\{
    \begin{array}{ll}
      \frac{\pi}{Z}k,   & 0\leq k\leq N_z/2, \\
       \frac{\pi}{Z}(N_z-k),  &  N_z/2+1\leq k\leq N_z-1.
    \end{array}\right. \nonumber
\end{align}

To approximate the nonlinear term $\mathcal{R}_{1}(\phi,u)$, we apply the spectral collocation method, which yields the discrete approximation $R_{1,h}(\phi,u) \approx \mathcal{R}_{1}(\phi,u)$. 

For the nonlinear term $\mathcal{R}_{2}(\phi,u)$, we use the second-order finite difference discretization introduced in Section~\ref{Sec:ETDforMOK} on the same uniform mesh and denote the approximation by $R_{2,h}(\phi,u) \approx \mathcal{R}_{2}(\phi,u)$. This choice is motivated by the surface diffusion operator $\Delta_{\mathrm{S}}$, whose implementation is difficult by using the same spectral collocation method. In particular, the nonlocal surface term $\Delta_{\mathrm{S}}^{-1}\bigl(g(\phi)(u-\bar{u})\bigr)$ is approximated by a narrow-band strategy. Since the localization function $g(\phi)$ is essentially supported only near the diffuse interface and is nearly zero away from the membrane represented by the phase-field function $\phi$, solving this nonlocal surface problem on the whole computational domain would introduce unnecessary computational cost. This narrow-band strategy restricts the nonlocal term to a thin band surrounding the membrane, defined by $ \Omega_{\mathrm{band}}:= \{x\in\Omega:\ |g(\phi(x))|\geq\delta\}$, where $0<\delta\ll 1$ is a small threshold. The nonlocal term is then assembled and solved only on the grid nodes in $\Omega_{\mathrm{band}}$. This narrow-band treatment reduces both the number of unknowns and the computational cost for evaluating the nonlocal term.

Next, we take the FFT on both sides 
of the equations, and by taking $\hat{\Phi} = \mathrm{FFT}(\Phi) = (\hat{\phi}_{ijk})_{0:N_x-1,0:N_y-1,0:N_z-1}$, $\hat{\mathbf{U}} = \mathrm{FFT}(\mathbf{U}) = (\hat{u}_{ijk})_{0:N_x-1,0:N_y-1,0:N_z-1}$ and $\widehat{\mathbf{GU}} =  \mathrm{FFT}(\mathbf{GU}) = (\widehat{g_{ijk}u_{ijk}})_{0:N_x-1,0:N_y-1,0:N_z-1}$, we then get 
\begin{align}
   & \hat{\Phi}_t = - \mathrm{L}_{\phi}\cdot\hat{\Phi} + \mathrm{R}_{1}(\hat{\Phi},\hat{\mathbf{U}}), \nonumber\\
   & \widehat{\mathbf{GU}}_t = - \mathrm{L}_{u}\cdot\widehat{\mathbf{GU}} + \mathrm{R}_{2}(\hat{\Phi},\hat{\mathbf{U}}). \nonumber
\end{align}
Here, the transformed linear terms are given as 
\begin{align}
    &-\mathrm{L}_{\phi}\cdot\hat{\Phi} = -(l^{\phi}_{ijk}\hat{\phi}_{ijk})_{0:N_x-1,0:N_y-1,0:N_z-1}, \nonumber\\
    & l^{\phi}_{ijk} = \frac{\kappa}{\mu}\lambda_{ijk}^2 - \frac{\kappa (A_1 + A_2) + \lambda_{\mathrm{surf}}\epsilon_{\phi}^2}{\mu\epsilon_{\phi}^2}\lambda_{ijk}+\frac{\kappa A_1A_2 + \lambda_{\mathrm{surf}}A_1\epsilon_{\phi}^2}{\mu\epsilon_{\phi}^4}, \nonumber \\
    &-\mathrm{L}_{u}\cdot\widehat{\mathbf{GU}} = -(l^{u}_{ijk}\widehat{g_{ijk}u_{ijk}})_{0:N_x-1,0:N_y-1,0:N_z-1}, \nonumber\\
    & l^{u}_{ijk} = B_2 - \epsilon_uB_1\lambda_{ijk}, \nonumber
\end{align}
and the transformed nonlinear terms are 
\begin{align}
    &\mathrm{R}_{1}(\hat{\Phi},\hat{\mathbf{U}}) = \mathrm{FFT}(R_{1,h}(\mathrm{iFFT}(\hat{\Phi}),\mathrm{iFFT}(\hat{\mathbf{U}}))), \nonumber \\
    &\mathrm{R}_{2}(\hat{\Phi},\hat{\mathbf{U}}) = \mathrm{FFT}(R_{2,h}(\mathrm{iFFT}(\hat{\Phi}),\mathrm{iFFT}(\hat{\mathbf{U}}))). \nonumber
\end{align}
Therefore, for $0\leq i\leq N_x-1$, $0\leq j\leq N_y-1$, $0\leq k \leq N_z-1 $, two equations can be written point-wisely as
\begin{align}
    & (\hat{\phi}_{ijk})_t = -(l^{\phi}_{ijk}\hat{\phi}_{ijk}) + [\mathrm{R}_{1}(\hat{\Phi},\hat{\mathbf{U}})]_{ijk}, \nonumber\\
    & (\widehat{g_{ijk}u_{ijk}})_t = -(l^{u}_{ijk}\widehat{g_{ijk}u_{ijk}}) + [\mathrm{R}_{2}(\hat{\Phi},\hat{\mathbf{U}})]_{ijk}. \nonumber
\end{align}

\subsection{Exponential Time Differencing Schemes}

Given the discrete function $g_{ijk} = 18(\phi_{ijk}^2 - \phi_{ijk})^2$, we then decouple the system and apply the ETD methods for the system. Denote the approximation solutions at $t = t_n$ step as $\hat{\Phi}^{n} = (\hat{\phi}^{n}_{ijk})$ and $\hat{\mathbf{U}}^n = (\hat{u}^n_{ijk})$, the ETD1 scheme is given by 
\begin{align}
    \hat{\Phi}^{n+1} = \ & \mathrm{ETD1}(\hat{\Phi}^{n},\hat{\mathbf{U}}^n,\tau_{n+1}, \mathrm{L}_{\phi},\mathrm{R}_{1}): \nonumber\\ 
    \hat{\phi}_{ijk}^{n+1} = \ & \varphi_{0}({\tau_{n+1}l_{ijk}^{\phi}})\hat{\phi}_{ijk}^{n} + \tau_{n+1} \varphi_1(\tau_{n+1}l^{\phi}_{ijk})[\mathrm{R}_{1}(\hat{\Phi}^{n},\hat{\mathbf{U}}^{n})]_{ijk}; \nonumber\\
    \widehat{\mathbf{GU}}^{n+1} = \ & \mathrm{ETD1}(\hat{\Phi}^{n+1},\hat{\mathbf{U}}^n,\tau_{n+1}, \mathrm{L}_{u},\mathrm{R}_{2}): \nonumber \\
    \widehat{g_{ijk}u_{ijk}}^{n+1} = \ & \varphi_{0}({\tau_{n+1}l^{u}_{ijk}})\mathrm{FFT}(g(\phi^{n+1})_{ijk}u^n_{ijk}) \nonumber\\
   & + \tau_{n+1} \varphi_1(\tau_{n+1}l^{u}_{ijk})[\mathrm{R}_{2}(\hat{\Phi}^{n+1},\hat{\mathbf{U}}^{n})]_{ijk}, \nonumber 
\end{align}
and the ETDRK2 scheme is given by
\begin{align}
     \hat{\Phi}^{n+1} = \ & \mathrm{ETDRK2}(\hat{\Phi}^{n},\hat{\mathbf{U}}^n,\tau_{n+1}, \mathrm{L}_{\phi},\mathrm{R}_{1}): \nonumber \\
     \Tilde{\Phi}^{n+1} =\ & (\Tilde{\phi}_{ijk}^{n+1}) =\mathrm{ETD1}(\hat{\Phi}^{n},\hat{\mathbf{U}}^n,\tau_{n+1}, \mathrm{L}_{\phi},\mathrm{R}_{1}), \nonumber \\
     \hat{\phi}_{ijk}^{n+1} = \ & \Tilde{\phi}_{ijk}^{n+1} + \tau_{n+1} \varphi_2(\tau_{n+1}l^{\phi}_{ijk})[\mathrm{R}_{1}(\Tilde{\Phi}^{n+1},\hat{\mathbf{U}}^{n}) - \mathrm{R}_{1}(\hat{\Phi}^{n},\hat{\mathbf{U}}^{n})]_{ijk}; \nonumber\\
     \widehat{\mathbf{GU}}^{n+1} = \ & \mathrm{ETDRK2}(\hat{\Phi}^{n+1},\hat{\mathbf{U}}^n,\tau_{n+1}, \mathrm{L}_{u},\mathrm{R}_{2}): \nonumber \\
     \widetilde{\mathbf{GU}}^{n+1} = \ &(\widetilde{g_{ijk}u_{ijk}}^{n+1}) = \mathrm{ETD1}(\hat{\Phi}^{n+1},\hat{\mathbf{U}}^n,\tau_{n+1}, \mathrm{L}_{u},\mathrm{R}_{2}), \nonumber\\
     \widehat{g_{ijk}u_{ijk}}^{n+1} = \ &\widetilde{g_{ijk}u_{ijk}}^{n+1} \nonumber \\
    & + \tau_{n+1} \varphi_2(\tau_{n+1}l^{u}_{ijk})[\mathrm{R}_{2}(\hat{\Phi}^{n+1},\Tilde{\mathbf{U}}^{n+1}) - \mathrm{R}_{2}(\hat{\Phi}^{n+1},\hat{\mathbf{U}}^{n})]_{ijk}. \nonumber 
\end{align}
Then we summarize the stabilized ETD1 and ETDRK2 schemes for the system \ref{eqn:model_force}-\ref{eqn:model_OK} in Algorithm~\ref{alg:ETD1_RK2}:

\begin{algorithm}
\caption{ETD1/ETDRK2 schemes}
\label{alg:ETD1_RK2}
\begin{algorithmic}
\STATE{Input: the computational domain $\Omega\subset \mathbb{R}^{3}$, number of grid points $N_x$, $N_y$, $N_z$, spatial grid sizes $h_x$, $h_y$, $h_z$, initial conditions $\Phi^0$ and $\mathbf{U}^0$, end time $\mathrm{T}>0$, and time step size $\tau$. Take $\hat{\Phi}^0 = \mathrm{FFT}(\Phi^0)$, $\mathbf{G}^0 = 18((\Phi^0)^2 - \Phi^0)^2$, $\hat{\mathbf{U}}^0 = \mathrm{FFT}(\mathbf{U}^0)$, and $K = T/\tau$.}
\WHILE{$n=0,1,\cdots,K-1$}
\STATE{$\hat{\Phi}^{n+1} = \mathrm{ETD1}/\mathrm{ETDRK2}(\hat{\Phi}^{n},\hat{\mathbf{U}}^n,\tau, \mathrm{L}_{\phi},\mathrm{R}_{1})$;}
\STATE{$\widehat{\mathbf{GU}}^{n+1} = \mathrm{ETD1}/\mathrm{ETDRK2}(\hat{\Phi}^{n+1},\hat{\mathbf{U}}^n,\tau, \mathrm{L}_{u},\mathrm{R}_{2})$;}
\STATE{$\Phi^{n+1} = \mathrm{iFFT}(\hat{\Phi}^{n+1})$;}
\STATE{$\mathbf{G}^{n+1} = 18((\Phi^{n+1})^2 - \Phi^{n+1})^2$;}
\STATE{$\mathbf{GU}^{n+1} = \mathrm{iFFT}(\widehat{\mathbf{GU}}^{n+1})$;}
\STATE{$\mathbf{U}^{n+1} = \frac{\mathbf{GU}^{n+1}}{\mathbf{G}^{n+1}}$.}
\ENDWHILE
\RETURN $\mathbf{\Phi}^{K},\mathbf{U}^{K}$
\end{algorithmic}
\end{algorithm}

Note that in the recovery step for $\mathbf{U}^{n+1}$, a narrow band $\Omega_{\mathrm{band}}^{n+1} = \{x\in\Omega_h:\ \mathbf{G}^{n+1}>10^{-3}\}$ is used to avoid division by very small values of the localization function. Specifically, for grid points inside this narrow band, $\mathbf{U}^{n+1}$ is computed explicitly by $\mathbf{U}^{n+1} = \frac{\mathbf{GU}^{n+1}}{\mathbf{G}^{n+1}}$. For grid points outside this narrow band, where $\mathbf{G}^{n+1}$ is negligible, we simply assign $\mathbf{U}^{n+1} = \mathbf{GU}^{n+1}$ to avoid numerical instability caused by division by a small number.

\section{Numerical Experiments}\label{Sec:NumExp}

In this section, we present some numerical examples to validate the proposed ETD schemes in Section \ref{Sec:ETDforMOK} and Section \ref{Sec:ETDforfull}. The experiments were implemented on a desktop with a 2.7 GHz Intel CPU and 32 GB of memory. 

\subsection{Convergence Rates for OK Model on Fixed Membranes}\label{subsec:rate}

In this section, we examine the convergence rates of the ETD1 \eqref{eqn:ETD1} and ETDRK2 \eqref{eqn:ETDRK2} schemes for the OK model on fixed membranes. We set the domain $\Omega = [-1,1]^2\subset\mathbb{R}^2$ with $N_x = N_y = N = 2^8$, and the mesh size is $h_x = h_y = h = \frac{2}{N} = \frac{1}{128}$. The fixed cell is given by 
\begin{align}
    \phi = 0.5+0.5\tanh{\left(\frac{3(r_0-r)}{\epsilon_{\phi}}\right)}, \quad r_0 = 0.4, \quad \epsilon_{\phi} = 10h, \nonumber
\end{align}
with $r$ the distance from each grid point to the center of the domain. The corresponding membrane is localized by $g(\phi) = 18(\phi^2-\phi)^2$. The initial condition for $u$ is then given as a localized protein-rich segment on the membrane. More precisely, on the membrane region identified by $g(\phi)$, we randomly select a connected subset $\Gamma_0 \subset \Omega_{\mathrm{mem}}= \{x\in\Omega:|g|\geq10^{-3}\}$ such that 
\begin{align}
    \frac{\int_{\Gamma_0}g(\phi)\mathrm{d}x}{\int_{\Omega_{\mathrm{mem}}}g(\phi)\mathrm{d}x} = \bar{u}+0.05 = 0.35, \nonumber
\end{align}
with the fixed protein fraction $\bar{u} = 0.3$. Then we set 
\begin{align}
    U^0 = \left\{
    \begin{array}{cc}
      1 , & x\in \Gamma_0,\\
      0 , & x\notin \Gamma_0,
    \end{array}
    \right. \nonumber
\end{align}
which represents an initial protein patch occupying $35\%$ of the membrane area. The parameters are chosen as $\gamma = 100$, $B_2 = 2000$, $M=600$ for the ETD1 scheme \eqref{eqn:ETD1}, and $M = 200$ for the ETDRK2 scheme \eqref{eqn:ETDRK2}. All simulations are carried out up to the final time $T = 0.02$, and the benchmark solution is computed using a small time step size $\tau = 1\times 10^{-6}$. For the ETD1 scheme \eqref{eqn:ETD1}, we then test the numerical errors and convergence rates in Table \ref{table:ETD1} for three interface values of $\epsilon_{u} = 5h, 10h,15h$, with the time step size successively halved from $10^{-4}$ to $10^{-4}/2^4$. Similarly, Table \ref{table:ETDRK2} presents the corresponding results (errors and convergence rates) for the ETDRK2 scheme \eqref{eqn:ETDRK2}. The observed rates of convergence for small time step size $\tau$ closely match the theoretical predictions.

\begin{table}[t]
\begin{center}
\begin{tabular}{ |c c c c c c c|  }
\hline
     & \multicolumn{2}{c}{$\epsilon_{u} = 5h$} & \multicolumn{2}{c}{$\epsilon_{u} = 10h$} & \multicolumn{2}{c|}{$\epsilon_{u} = 15h$}\\
 \hline
 $\tau$ & Error \quad & Rate \quad & Error \quad & Rate \quad & Error \quad & Rate\quad \\
\hline
$1\mathrm{e}-4/2^0$ & 2.477e-1 \quad &  -    \quad      
                    & 3.230e-1 \quad &  -    \quad 
                    & 3.374e-1 \quad &  -    \quad \\
$1\mathrm{e}-4/2^1$ & 1.832e-1 \quad & 0.435 \quad 
                    & 1.903e-1 \quad & 0.763 \quad  
                    & 1.778e-1 \quad & 0.924 \quad \\
$1\mathrm{e}-4/2^2$ & 1.204e-1 \quad & 0.606 \quad 
                    & 1.084e-1 \quad & 0.811 \quad 
                    & 9.207e-2 \quad & 0.950 \quad \\
$1\mathrm{e}-4/2^3$ & 6.870e-2 \quad & 0.810 \quad 
                    & 5.713e-2 \quad & 0.924 \quad  
                    & 4.625e-2 \quad & 0.993 \quad \\
$1\mathrm{e}-4/2^4$ & 3.450e-2 \quad & 0.994 \quad 
                    & 2.740e-2 \quad & 1.060 \quad  
                    & 2.165e-2 \quad & 1.095 \quad \\
\hline
\end{tabular}
\end{center}
\caption{The discrete $L^{\infty}$ errors and the corresponding rate of convergence at time $T=0.02$ generated by the ETD1 scheme \eqref{eqn:ETD1} with different interface values of $\epsilon_{u}$. The parameters given in examples are $\bar{u} = 0.3$, $M=600$, $\gamma = 100$, and $B_2 = 2000$.}
\label{table:ETD1}
\end{table}

\begin{table}[t]
\begin{center}
\begin{tabular}{ |c c c c c c c|  }
\hline
     & \multicolumn{2}{c}{$\epsilon_{u} = 5h$} & \multicolumn{2}{c}{$\epsilon_{u} = 10h$} & \multicolumn{2}{c|}{$\epsilon_{u} = 15h$}\\
 \hline
 $\tau$ & Error \quad & Rate \quad & Error \quad & Rate \quad & Error \quad & Rate\quad \\
\hline
$1\mathrm{e}-4/2^0$ & 7.890e-2 \quad &  -    \quad      
                    & 7.473e-2 \quad &  -    \quad 
                    & 7.438e-2 \quad &  -    \quad \\
$1\mathrm{e}-4/2^1$ & 4.521e-2 \quad & 0.803 \quad 
                    & 3.147e-2 \quad & 1.247 \quad  
                    & 2.502e-2 \quad & 1.572 \quad \\
$1\mathrm{e}-4/2^2$ & 2.059e-2 \quad & 1.134 \quad 
                    & 1.118e-2 \quad & 1.493 \quad 
                    & 7.705e-3 \quad & 1.699 \quad \\
$1\mathrm{e}-4/2^3$ & 7.321e-3 \quad & 1.492 \quad 
                    & 3.383e-3 \quad & 1.725 \quad  
                    & 2.163e-3 \quad & 1.833 \quad \\
$1\mathrm{e}-4/2^4$ & 2.178e-3 \quad & 1.749 \quad 
                    & 9.179e-4 \quad & 1.882 \quad  
                    & 5.658e-4 \quad & 1.935 \quad \\
\hline
\end{tabular}
\end{center}
\caption{The discrete $L^{\infty}$ errors and the corresponding rate of convergence at time $T=0.02$ generated by the ETDRK2 scheme \eqref{eqn:ETDRK2} with different interface values of $\epsilon_{u}$. The parameters given in examples are $\bar{u} = 0.3$, $M=200$, $\gamma = 100$, $B_2 = 2000$.}
\label{table:ETDRK2}
\end{table}

\subsection{2D Coarsening Dynamics of ETD1 and ETDRK2 Schemes for OK Model on Fixed Membranes}

In this section, we verify the discrete MBP and energy stability of the ETD1 \eqref{eqn:ETD1} and ETDRK2 \eqref{eqn:ETDRK2} schemes for the OK model on fixed membranes in two dimensions, and we show that these properties persist for complex surface geometries generated by the localization function $g(\phi)$. All simulations use the same computational domain $\Omega$ and grid number $N$ as in Section \ref{subsec:rate}. Unless stated otherwise, the final time is $T = 1000$, with parameters $\bar{u} = 0.5$, $\epsilon_{u} = 5h$, $\tau = 10^{-3}$, and the stabilizer $B_2 = 2000$ (sufficiently large to fulfill the condition \eqref{eqn:cond_MBP} and \eqref{eqn:cond_ES}).

Figures \ref{fig:ETD1_cir8} and \ref{fig:ETDRK2_circ9} correspond to a fixed circular phase-field cell defined by 
\begin{align}
    \phi = 0.5+0.5\tanh{\left(\frac{3(r_0-r)}{\epsilon_{\phi}}\right)}, \quad r_0 = 0.4, \quad \epsilon_{\phi} = 10h, \nonumber
\end{align}
with $r$ the distance from each grid point to the center of the domain and the membrane indicator $g(\phi)$. 
The initial condition is randomly generated by a coarse grid on the band region $\Omega_{\mathrm{band}}$. Figure \ref{fig:ETD1_cir8} shows the coarsening dynamics produced by ETD1 \eqref{eqn:ETD1} with repulsion strength $\gamma = 2000$: starting from a random initial ($t=0$) the system quickly phase separates into eight protein-rich domains (red patches) of different sizes (t=10), and then relaxes to eight regions that are nearly equal in size and spacing along the membrane at later times ($t=500, 1000$). The four insets display snapshots at $t=0,10,500,1000$. The blue curve (left axis) shows the discrete energy $E_{\phi,h}[U^n]$ \eqref{eqn:Energy_dis}, which decreases monotonically in time, while the orange curve (right axis) plots $\|2U^n-1\|_{\infty}$, with the bound $\|2U^n-1\|_{\infty}\leq 1$ is equivalent to $0\leq U^n\leq1$. This example confirms the discrete MBP and energy stability for the ETD1 scheme \eqref{eqn:ETD1}. The result in Figure \ref{fig:ETDRK2_circ9} generated by the ETDRK2 scheme \eqref{eqn:ETDRK2}, with $\gamma = 3000$, exhibits similar coarsening dynamics on the same circular membrane, but with the final state of nine equally spaced protein-rich domains. This indicates that a larger repulsive strength $\gamma$ promotes the phase separation of proteins. Again, the energy decay and the uniform bound on $U^n$ demonstrate the discrete MBP and energy stability of the ETDRK2 scheme \eqref{eqn:ETDRK2}.

We then test the ETDRK2 scheme \eqref{eqn:ETDRK2} on fixed membranes with more complex geometries: an elliptical cell (Figure \ref{fig:ETDRK2_ell9}) and a nearly circular cell with seven protrusions (Figure \ref{fig:ETDRK2_sin7}). In Figure \ref{fig:ETDRK2_ell9}, the elliptical cell is given by 
\begin{align}
    \phi = 0.5+0.5\tanh{\left(\frac{3(r_0-r)}{\epsilon_{\phi}}\right)}, \quad r_0 = 0.4, \quad r = \sqrt{x^2+\sqrt{3}y^2}, \nonumber
\end{align}
with the same membrane function $g(\phi)$. Starting from a random initial and using $\gamma=5000$, the system again enters similar coarsening dynamics and evolves to a pattern with nine equally-separated protein-rich domains along the elliptical membrane, while the energy and $\|2U^{n}-1\|_{\infty}$ behave similarly as 
in Figure \ref{fig:ETD1_cir8} and \ref{fig:ETDRK2_circ9}. Figure \ref{fig:ETDRK2_sin7} shows an example on a “wavy” round cell with seven protrusions as
\begin{align}
    \phi = 0.5+0.5\tanh{\left(\frac{3(r_0-r)}{\epsilon_{\phi}}\right)}, \quad r_0 = 0.4(1+0.1\sin{7\theta}), \quad \theta = \arctan{(y,x)}, \nonumber
\end{align}
again with random initial and $\gamma = 7000$. The dynamics produce fourteen protein-rich domains that arrange uniformly along the corrugated membrane, and the plots confirm monotone energy decay and the discrete MBP.

The experiments in Figures \ref{fig:ETD1_cir8}–\ref{fig:ETDRK2_sin7} show that, for a broad class of membrane geometries generated by the localization function $g(\phi)$, the ETD1 and ETDRK2 schemes preserve the discrete MBP and energy stability in practice, which match the theoretical results in Section \ref{Sec:ETDforMOK}.

\begin{figure}[t]
    \begin{center}
    \includegraphics[width=0.9\textwidth]{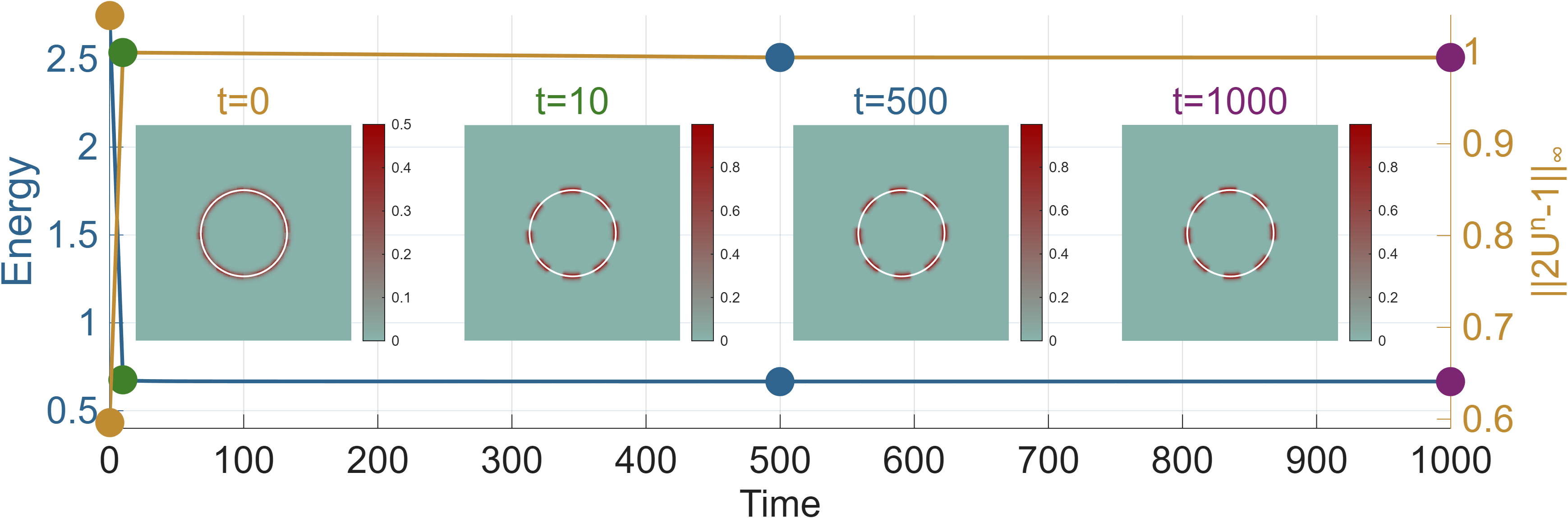}
    \end{center}
    \vspace{-3mm}
    \caption{2D coarsening dynamics process using ETD1 scheme \eqref{eqn:ETD1}, parameters $T=1000$, $\bar{u} = 0.5$, $\gamma = 2000$, $\tau = 10^{-3}$, and $B_2 = 2000$.}
    \label{fig:ETD1_cir8}
\end{figure}

\begin{figure}[t]
    \begin{center}
    \includegraphics[width=0.9\textwidth]{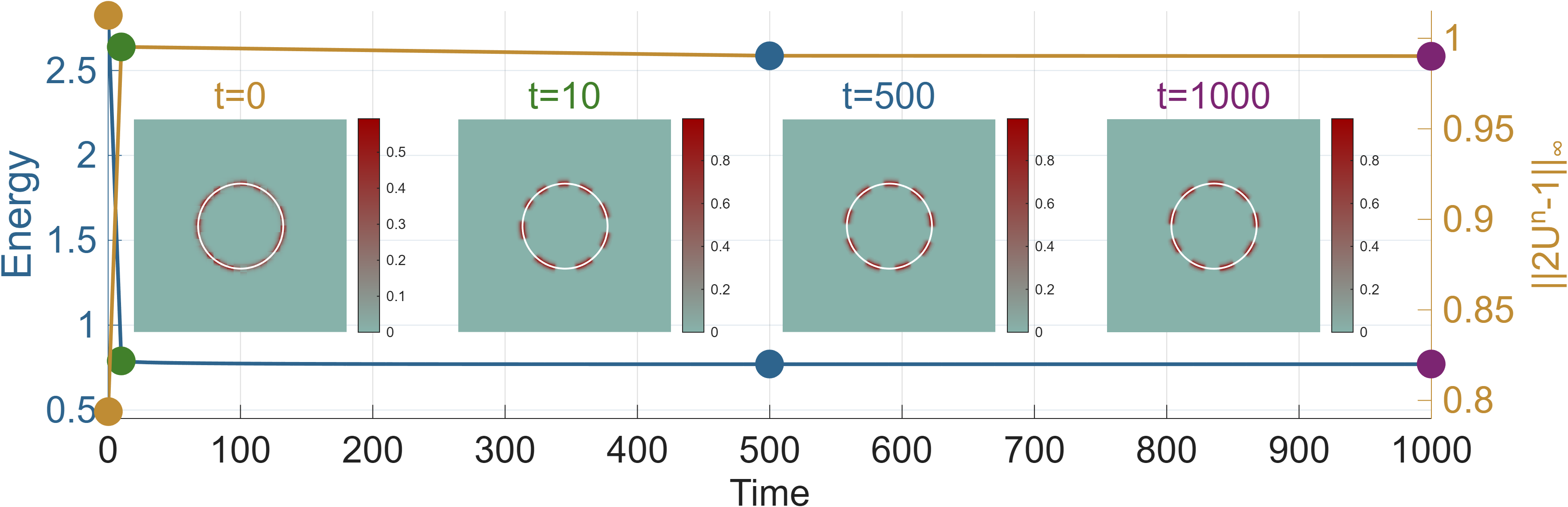}
    \end{center}
    \vspace{-3mm}
    \caption{2D coarsening dynamics process using ETDRK2 scheme \eqref{eqn:ETDRK2}, parameters $T=1000$, $\bar{u} = 0.5$, $\gamma = 3000$, $\tau = 10^{-3}$, and $B_2 = 2000$.}
    \label{fig:ETDRK2_circ9}
\end{figure}

\begin{figure}[t]
    \begin{center}
    \includegraphics[width=0.9\textwidth]{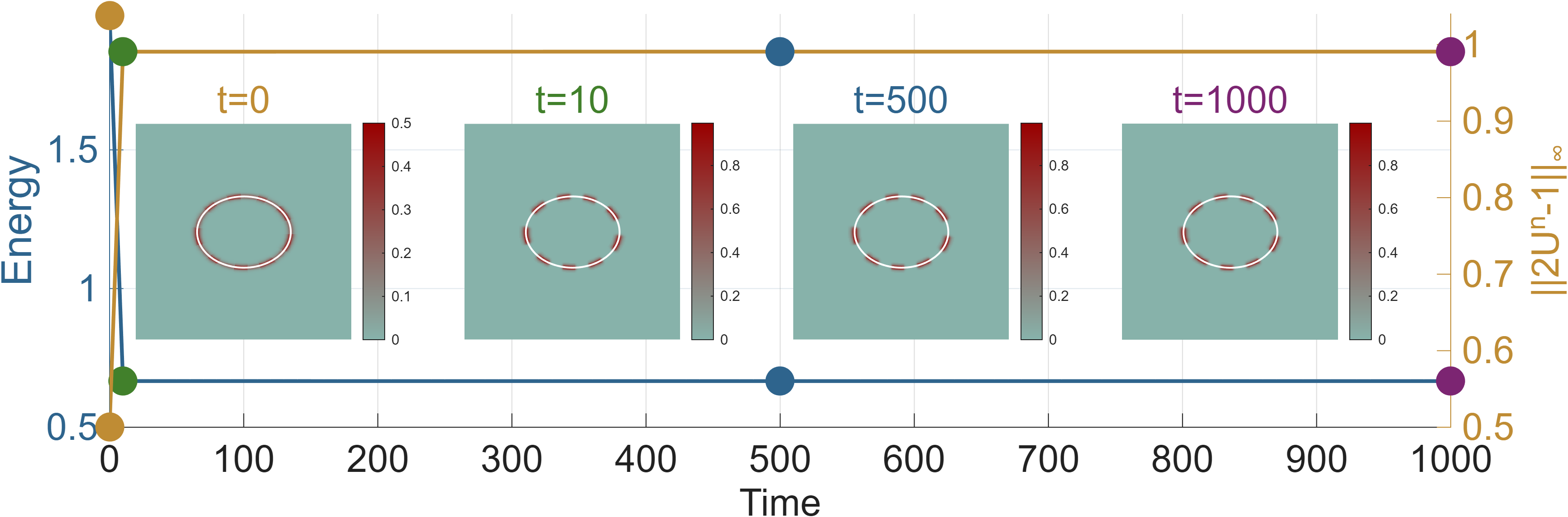}
    \end{center}
    \vspace{-3mm}
    \caption{2D coarsening dynamics process using ETDRK2 scheme \eqref{eqn:ETDRK2}, parameters $T=1000$, $\bar{u} = 0.5$, $\gamma = 5000$, $\tau = 10^{-3}$, and $B_2 = 2000$.}
    \label{fig:ETDRK2_ell9}
\end{figure}

\begin{figure}[t]
    \begin{center}
    \includegraphics[width=0.9\textwidth]{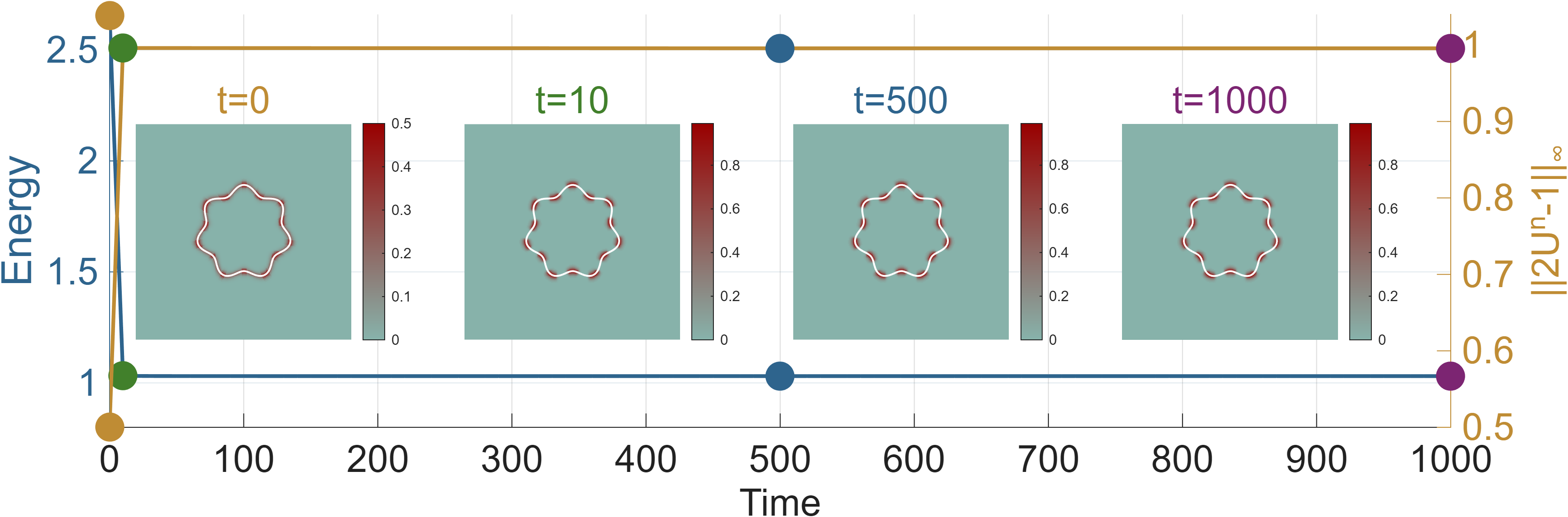}
    \end{center}
    \vspace{-3mm}
    \caption{2D coarsening dynamics process using ETDRK2 scheme \eqref{eqn:ETDRK2}, parameters $T=1000$, $\bar{u} = 0.5$, $\gamma = 7000$, $\tau = 10^{-3}$, and $B_2 = 2000$.}
    \label{fig:ETDRK2_sin7}
\end{figure}

\subsection{3D Dynamics for the Phase-field Model of Multicomponent Membranes}

In this section, we simulate the coupled system \eqref{eqn:model_force}-\eqref{eqn:model_OK} of multicomponent membranes in three dimensions by using Algorithm \ref{alg:ETD1_RK2}. The computational domain is chosen as $\Omega = [-1,1]^3\subset\mathbb{R}^3$ with a uniform grid $N = N_x = N_y = N_z = 2^7$, so that the mesh size is $h = h_x = h_y = h_z = \frac{2}{2^7} = \frac{1}{64}$. For the force-balance equation \eqref{eqn:force_spec}, we fix $\epsilon_{\phi} = 10h$, $\epsilon_{u} = 5h$, $\mu = 2$, $\lambda_{\mathrm{surf}} = 10$, $\lambda_{\mathrm{line}} = 30$, $\kappa = 1$, $M_{\mathrm{area}} = 100$, and the stabilizers $A_1 = 20$, $A_2=50$, and for the membrane-associated OK equation \eqref{eqn:model_OK}, we take $M=3000$ and the stabilizer $B_1 = 1$, $B_2 = 1000$. The time step size is fixed at $\tau = 10^{-3}$, and other parameters will be varied in the following experiments. 

Figure \ref{fig:ETDRK2_3d_gamma} shows the 3D coarsening dynamics of the proteins on the membrane and the associated membrane deformation. From left to right, we display snapshots at increasing times; from top to bottom, we increase the repulsive strength $\gamma$ to enhance protein segregation. The initial phase-field cell is 
\begin{align}
        \Phi^0 = 0.5+0.5\tanh{\left(\frac{3(r_0-r)}{\epsilon_{\phi}}\right)}, \quad r_0 = 0.4, \nonumber
\end{align}
where $r$ is the distance from each grid point to the center of the domain. The initial condition for $u$ is randomly generated on a coarse grid. On the top row of Figure \ref{fig:ETDRK2_3d_gamma}, we chose $\gamma = 15000$ and the biochemical strength $\alpha = 700$. Starting from a random initial of $u$ at $t=0$, the experiment first presents the protein segregation on the fixed membrane induced by $\Phi^0$: several protein-rich patches shown in red appear at $t=1$, and then evolve to nearly equal-sized and equal-distanced domains by $t=10$ and $t=100$. After the protein pattern has formed, the membrane begins to deform under the combined mechanical and biochemical forces, leading to shape changes at $t=101$ and $t=105$. On the bottom row, we increase $\gamma$ from $15000$ to $17000$. The overall dynamics are similar, but the stronger repulsion produces more protein-rich patches, confirming that larger $\gamma$ promotes microphase separation. 

In both cases in Figure \ref{fig:ETDRK2_3d_gamma}, the membrane develops outward protrusions (bumps) at the protein-rich regions. These bumps are not identical in height or shape. This is because, for a given choice of $\gamma$, the number of protein domains may not allow a perfectly uniform spacing on the closed surface. As a result, the distances between neighboring patches vary slightly, and the strength of their interactions is not identical. Protein-rich domains that are closer to their neighbors experience stronger repulsion than those in more isolated regions, which leads to small differences in local curvature, bump amplitude, and detailed shape. Hence, the protrusions in the last column of Figure \ref{fig:ETDRK2_3d_gamma} are not perfectly uniform across the membrane. Furthermore, this experiment shows a clear difference in time scales between protein segregation and membrane deformation. The protein needs a relatively long time (about $t = 100$) to reach a quasi-equilibrium pattern under the membrane-associated OK dynamics \eqref{eqn:model_OK}, while the subsequent membrane reshaping driven by the force-balance phase field equation \eqref{eqn:model_force} occurs on a shorter time scale (about $t = 5$) once the protein pattern is established. This suggests that, in this parameter regime, the OK dynamics on the evolving surface are the main computational cost, whereas the membrane deformation responds more quickly to the established protein forces.

Next, we test how the biochemical strength $\alpha$ affects cell morphology when the membrane carries multiple uniformly spaced protein-rich subdomains, as shown in Figure \ref{fig:ETDRK2_3d_12b}. For a fixed repulsive strength $\gamma = 4000$ and a random initial condition of $u$, the equilibrium state of the membrane-associated OK model \eqref{eqn:model_OK} on the surface induced by $\Phi^0$ contains $12$ protein-rich patches that are nearly uniformly distributed on the membrane and experience comparable interactions with their neighbors, which is displayed in the leftmost subfigure of Figure \ref{fig:ETDRK2_3d_12b}. Using this state as the initial condition, we then vary the biochemical strength $\alpha = 250, 300, 350, 500$, which generates the second to fifth subfigures. The simulations show that, under these conditions, the amplitudes and shapes of the bumps at the protein-rich subdomains remain almost identical, consistent with the uniform separation of the domains. At the same time, increasing $\alpha$ strengthens the protein-driven biochemical force and leads to stronger outward protrusions and larger membrane deformation. The resulting 3D morphologies in Figure \ref{fig:ETDRK2_3d_12b} closely resemble the earlier numerical results in \cite{wang2008modelling} and the experimental observations of multicomponent vesicles in \cite{baumgart2003imaging}.

\begin{figure}[t]
    \begin{center}
    \includegraphics[width=0.15\textwidth] {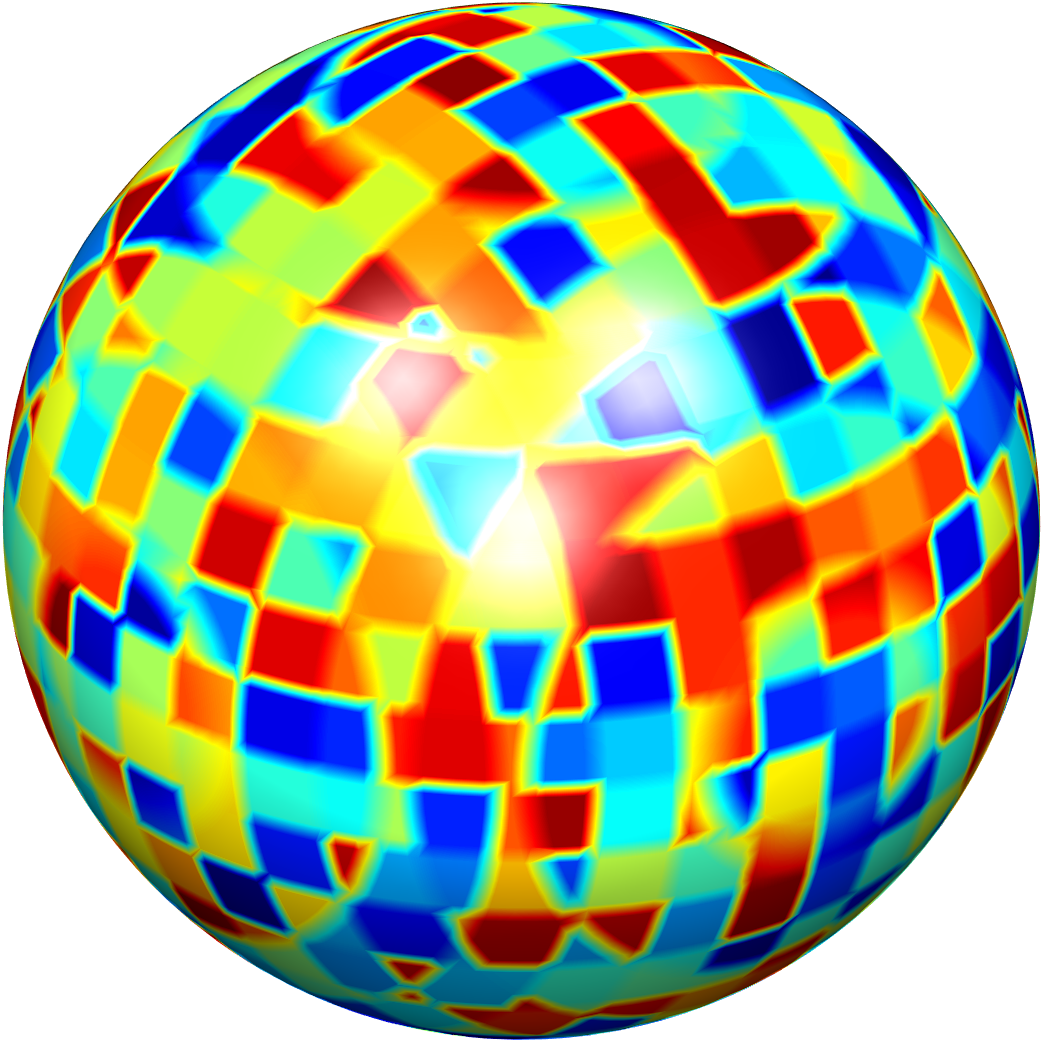} \
    \includegraphics[width=0.15\textwidth]{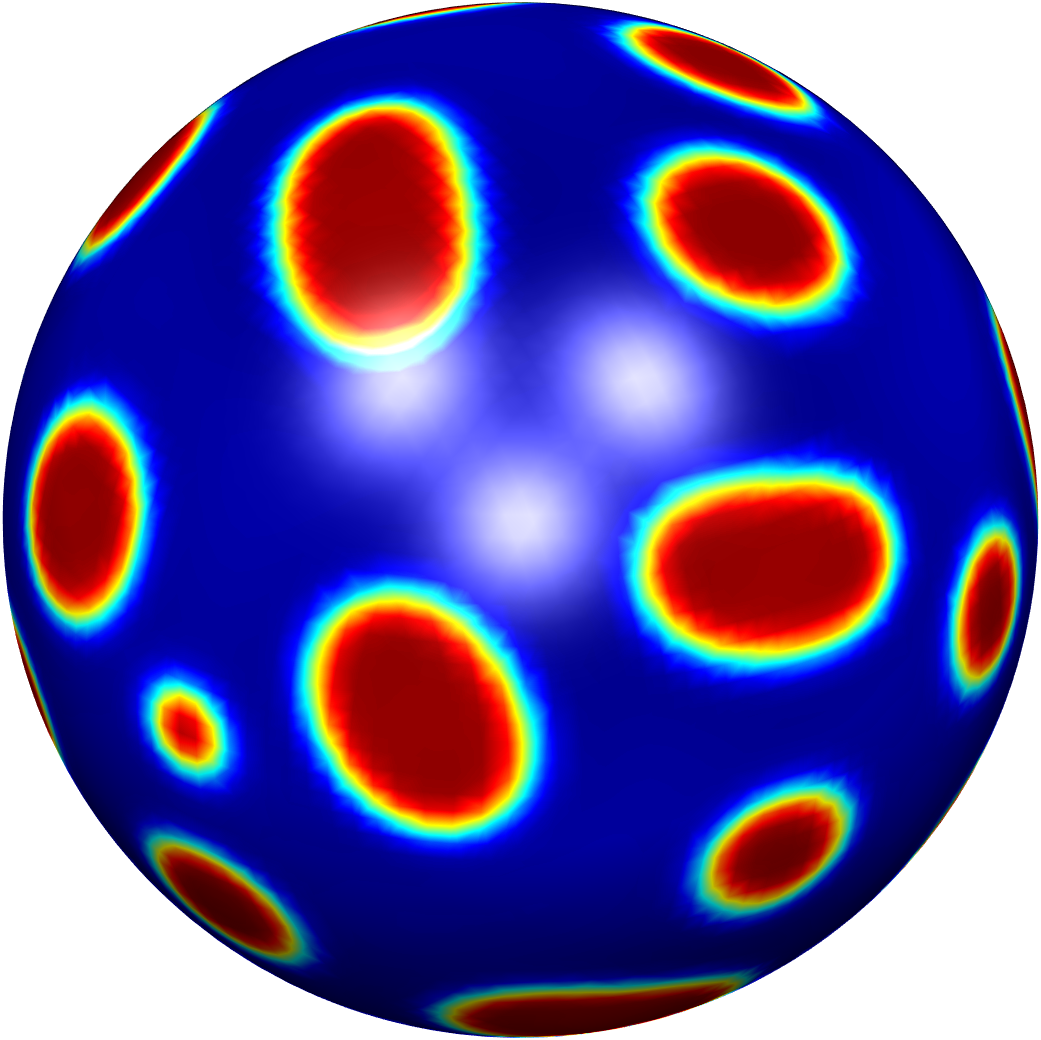} \
    \includegraphics[width=0.15\textwidth]{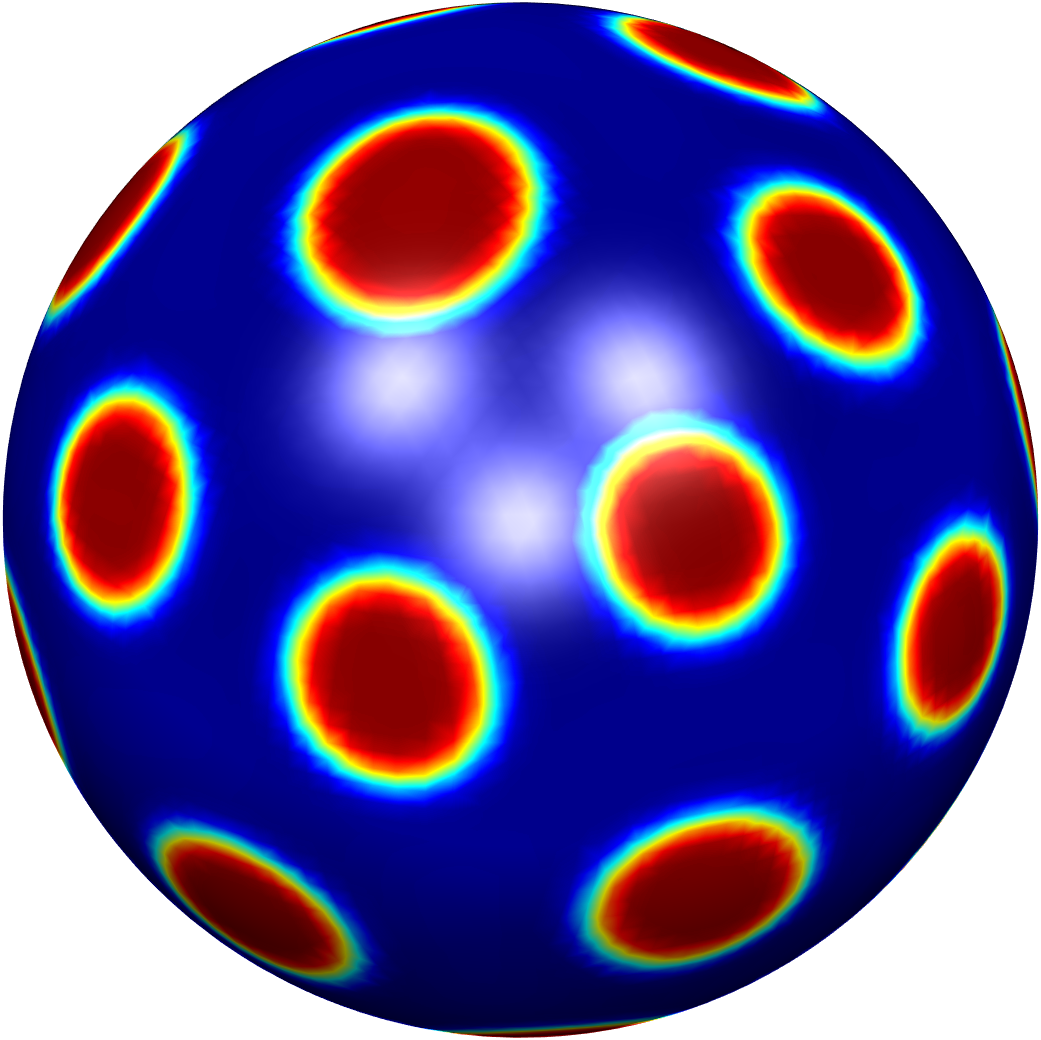} \
    \includegraphics[width=0.15\textwidth]{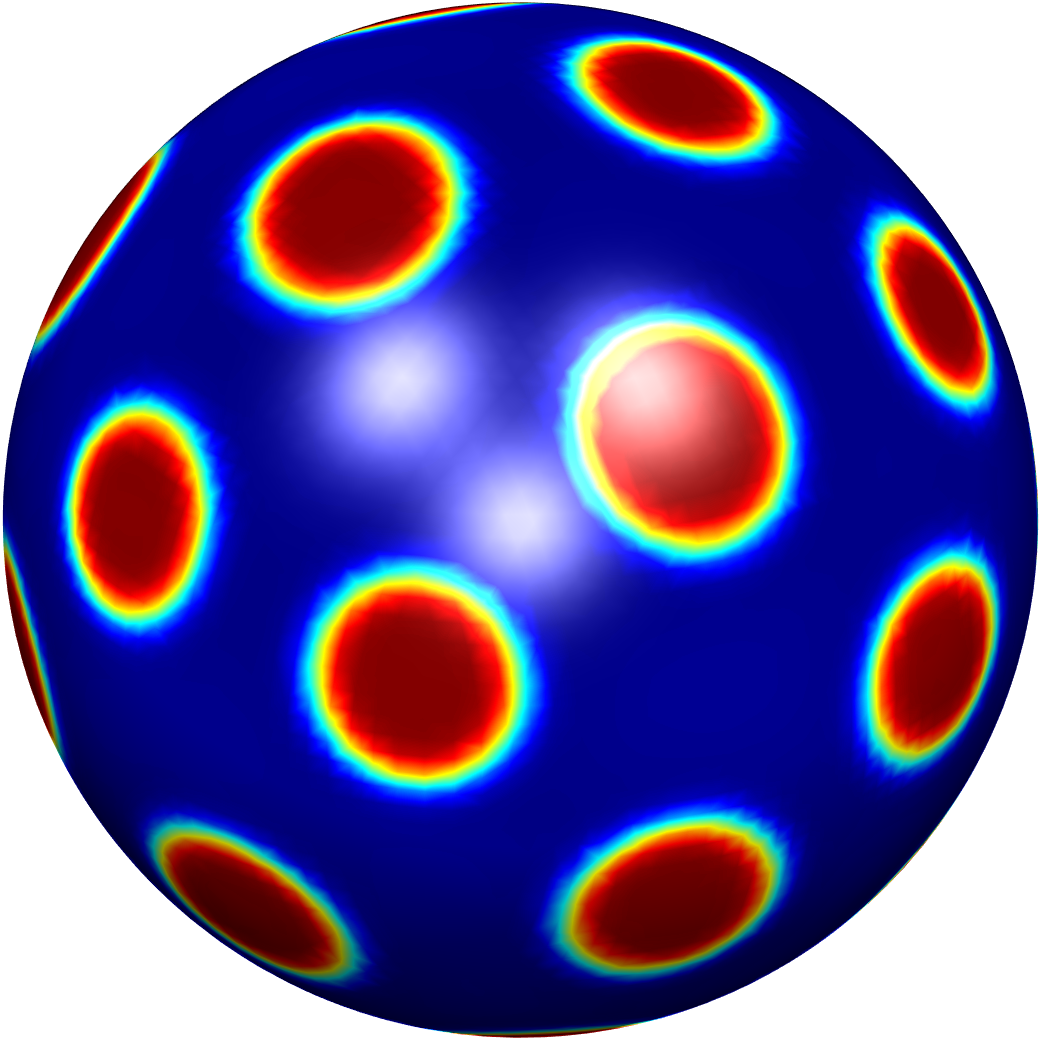} \
    \includegraphics[width=0.15\textwidth]{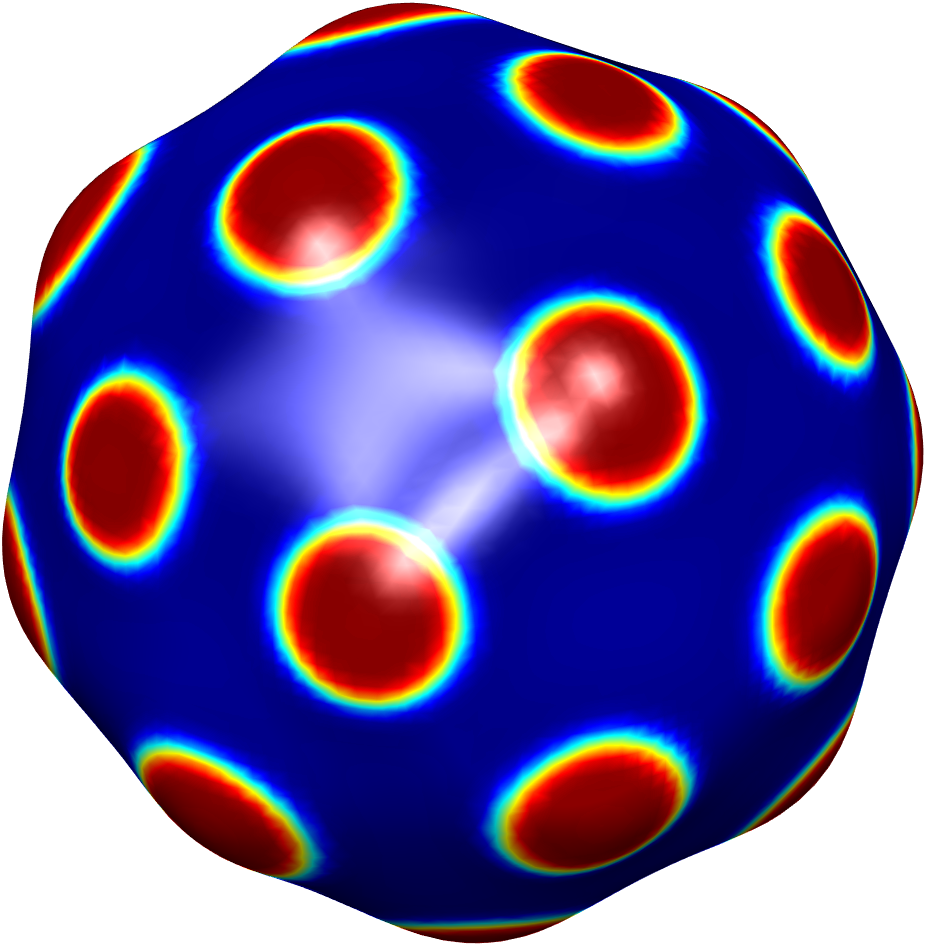} \
    \includegraphics[width=0.15\textwidth]{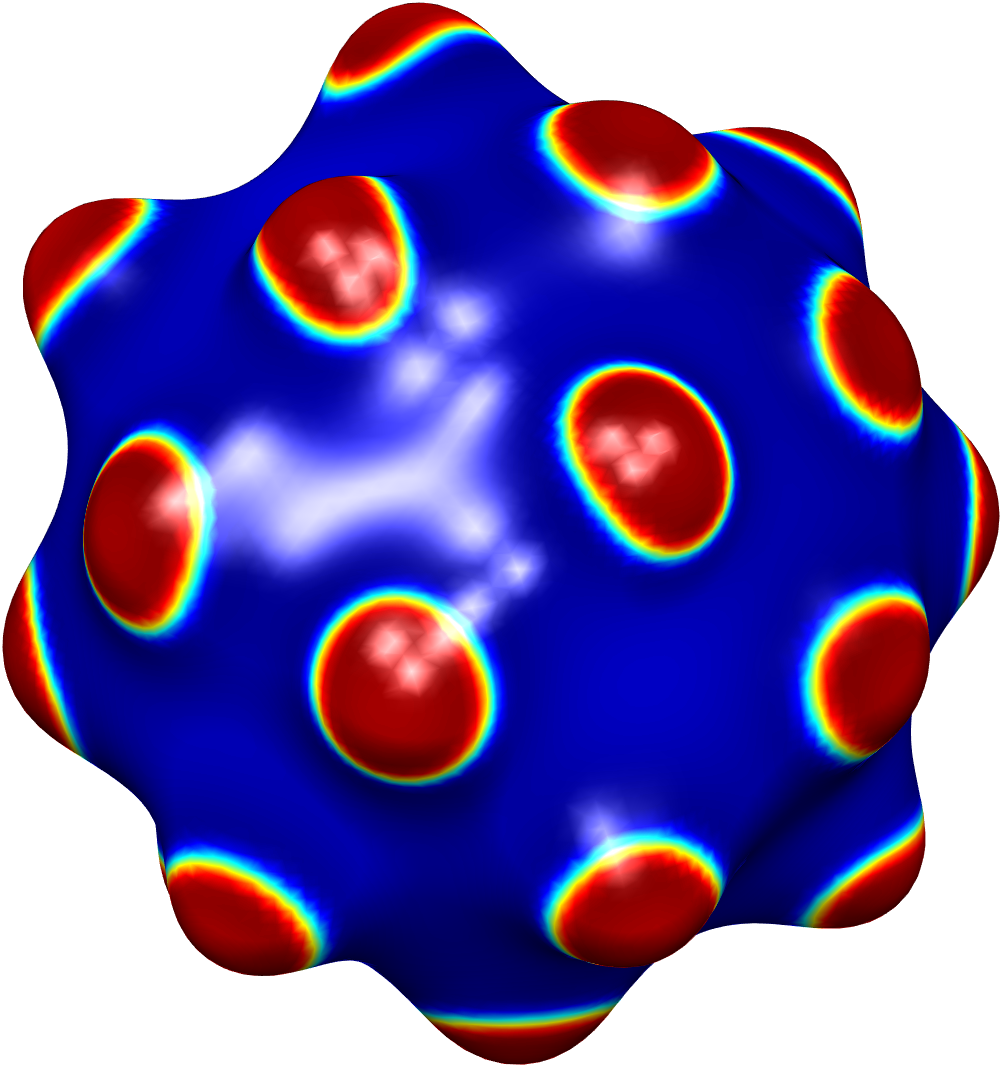} 
    \end{center}
    \begin{center}
    \includegraphics[width=0.15\textwidth] {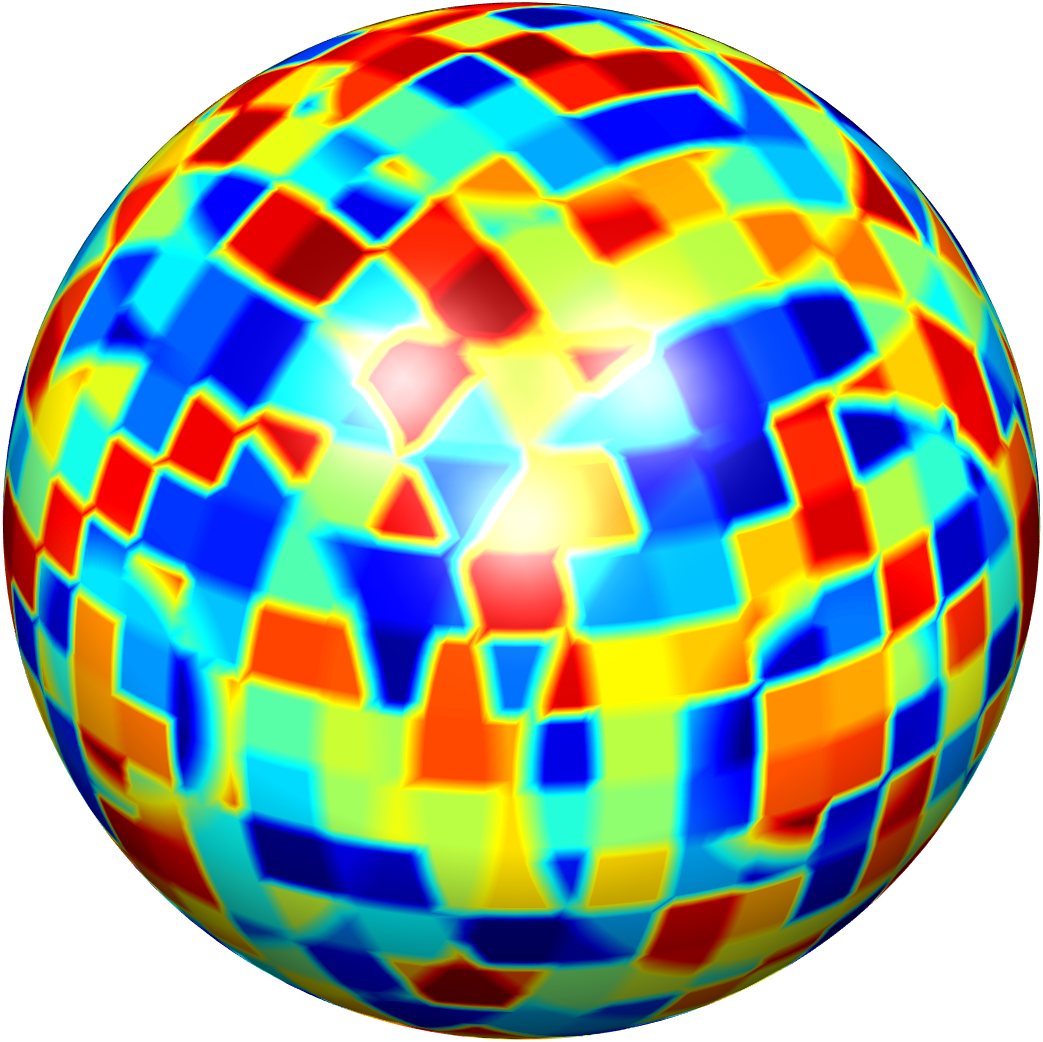} \
    \includegraphics[width=0.15\textwidth]{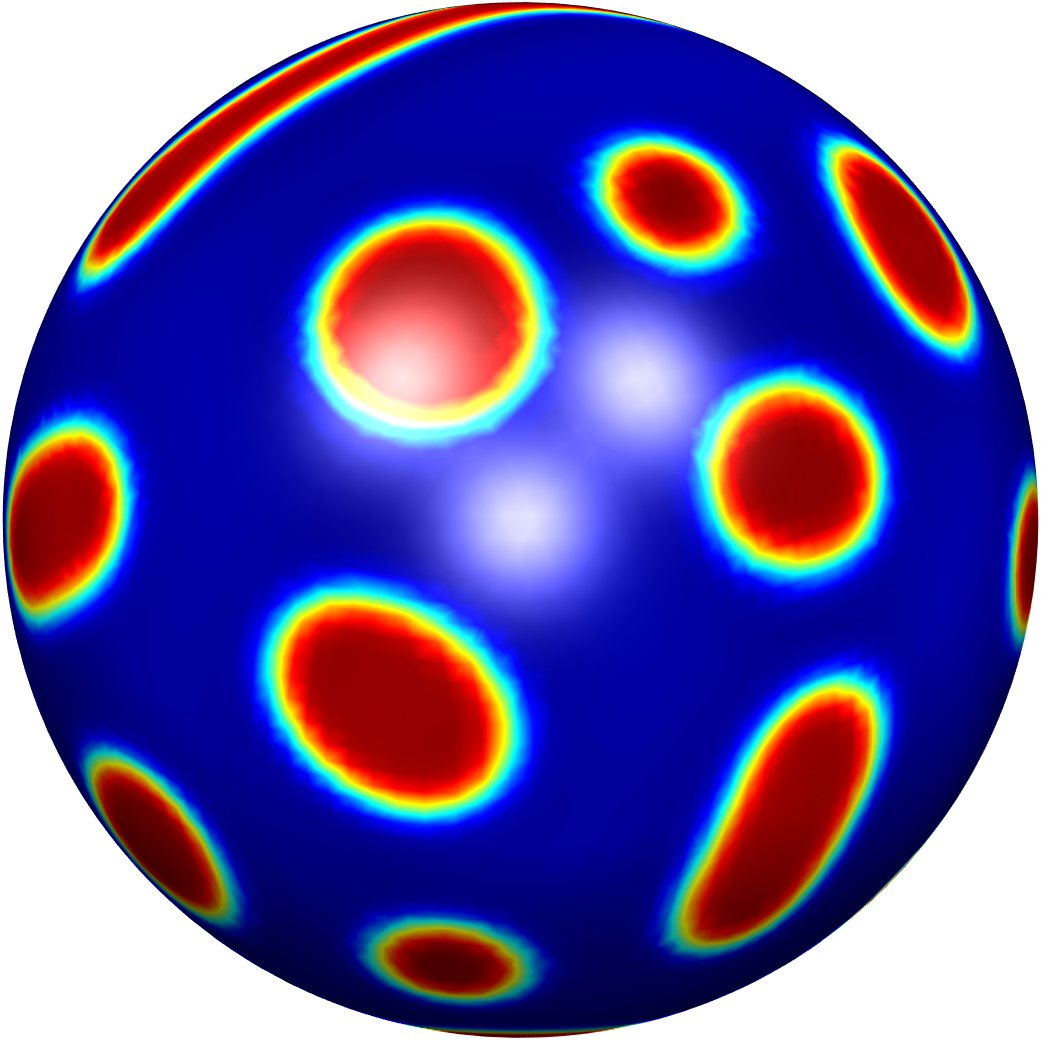} \
    \includegraphics[width=0.15\textwidth]{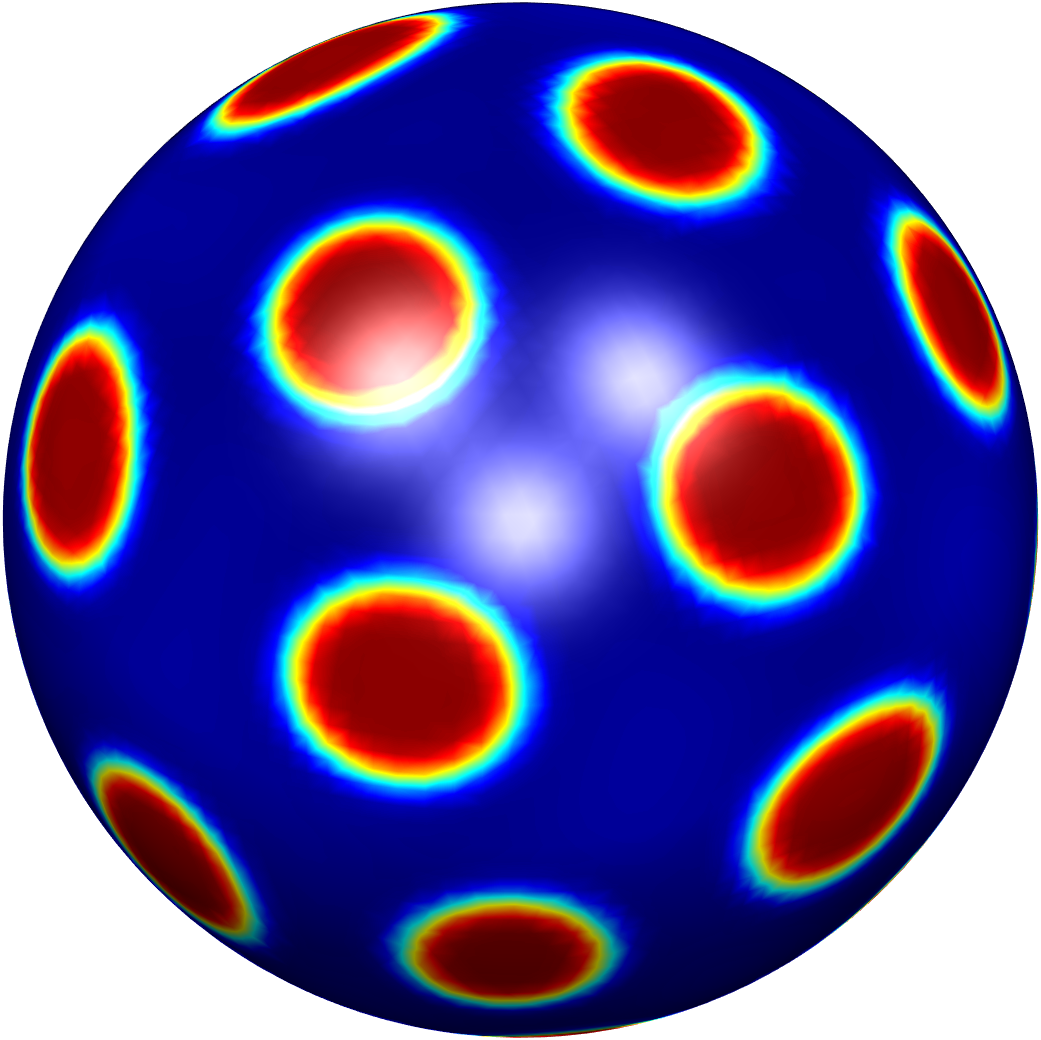} \
    \includegraphics[width=0.15\textwidth]{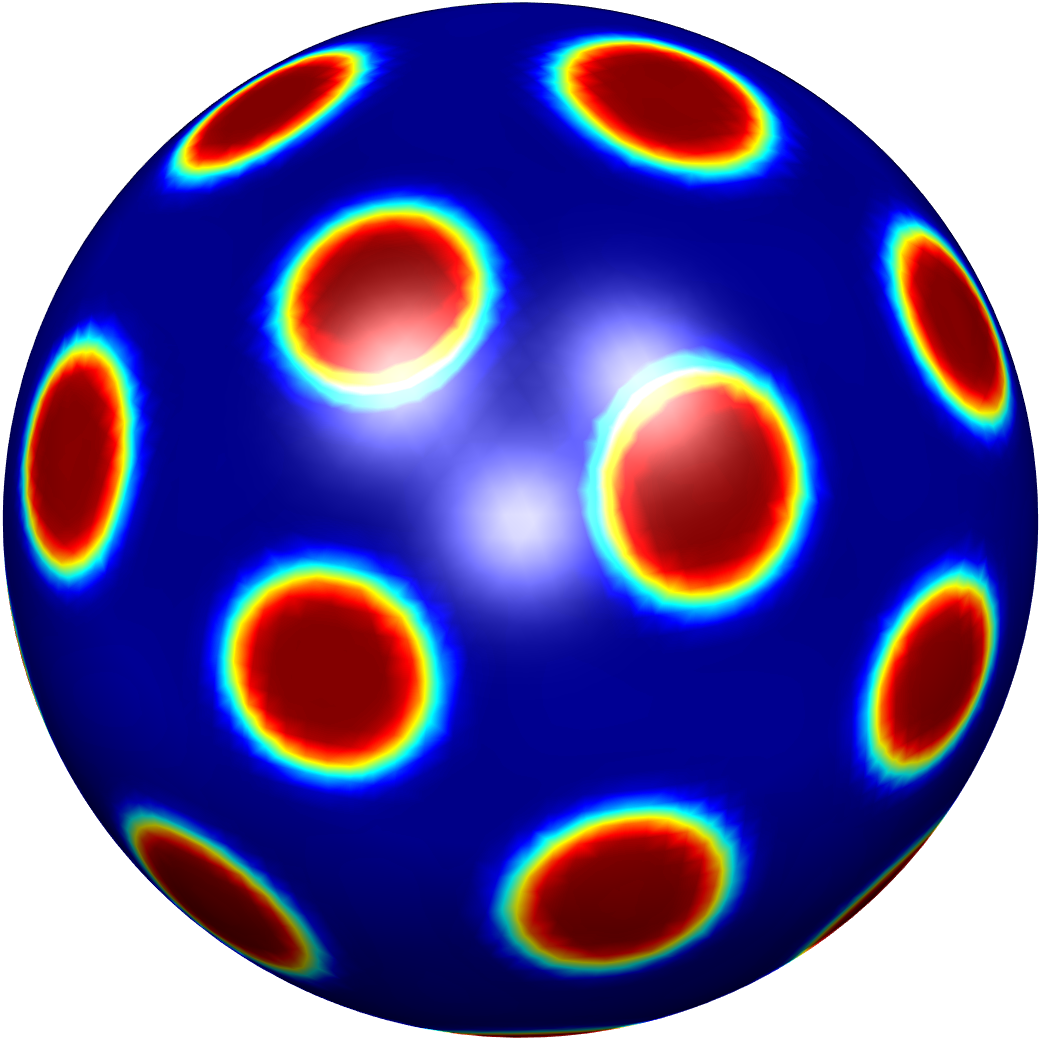} \
    \includegraphics[width=0.15\textwidth]{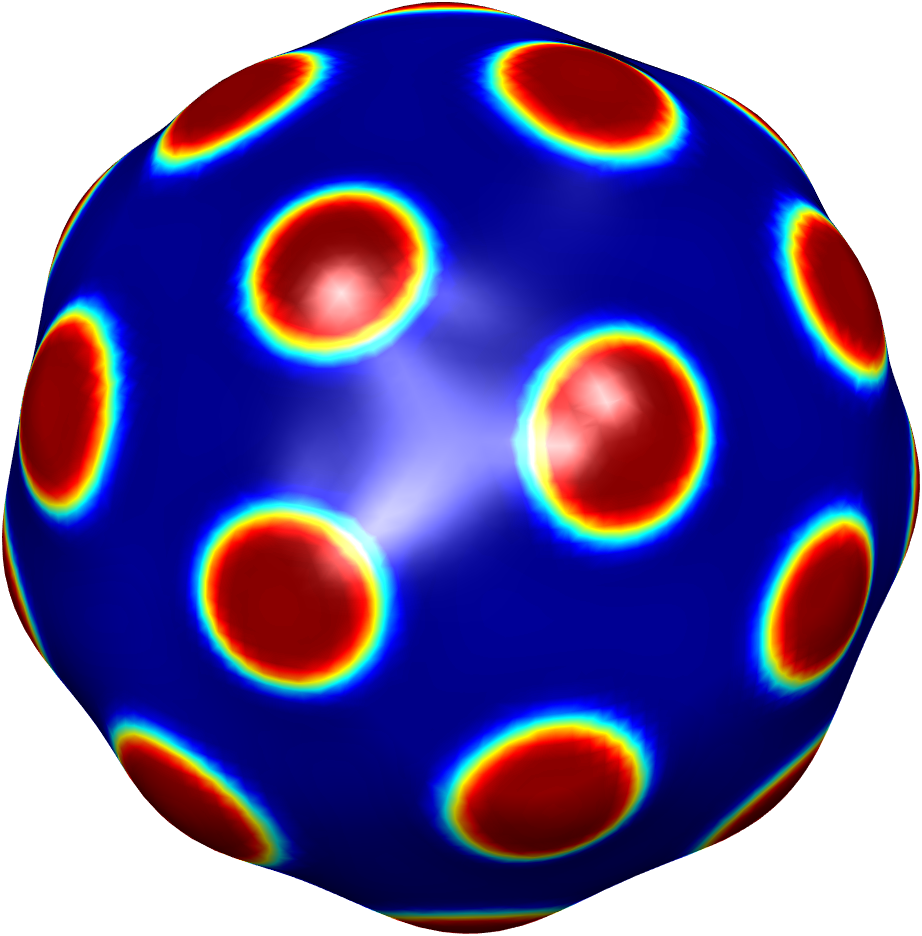} \
    \includegraphics[width=0.15\textwidth]{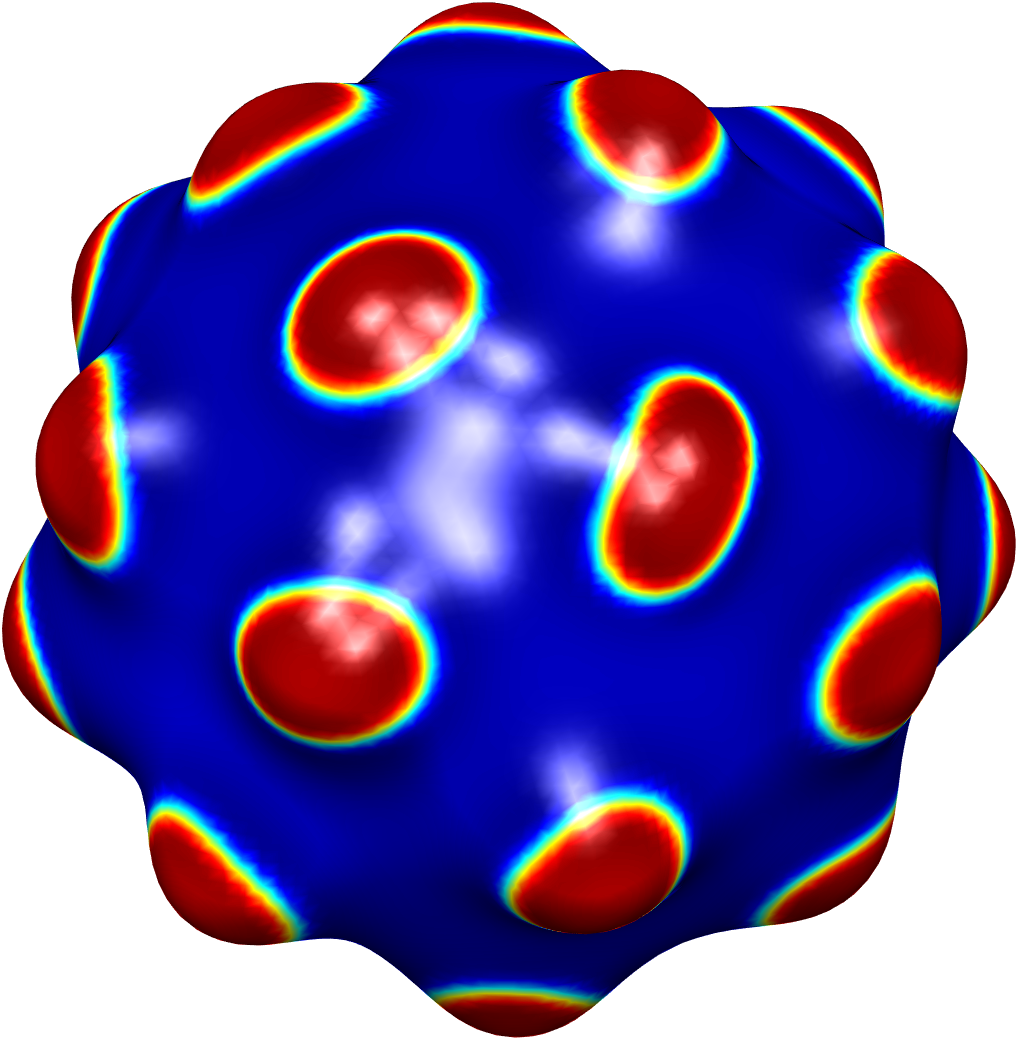}
    \end{center}
    \caption{3D coarsening dynamics process using Algorithm \ref{alg:ETD1_RK2}, parameters $\bar{u} = 0.3$, $\gamma = 15000$ (Top row), $\gamma = 17000$ (Bottom row), $\alpha = 700$, and time step size $\tau = 10^{-3}$. From left column to the right, snapshots at time $t=0$, $t=1$, $t=10$, $t=100$, $t=101$, $t=105$.}
    \label{fig:ETDRK2_3d_gamma}
\end{figure}

\begin{figure}[t]
    \begin{center}
    \includegraphics[width=0.18\textwidth] {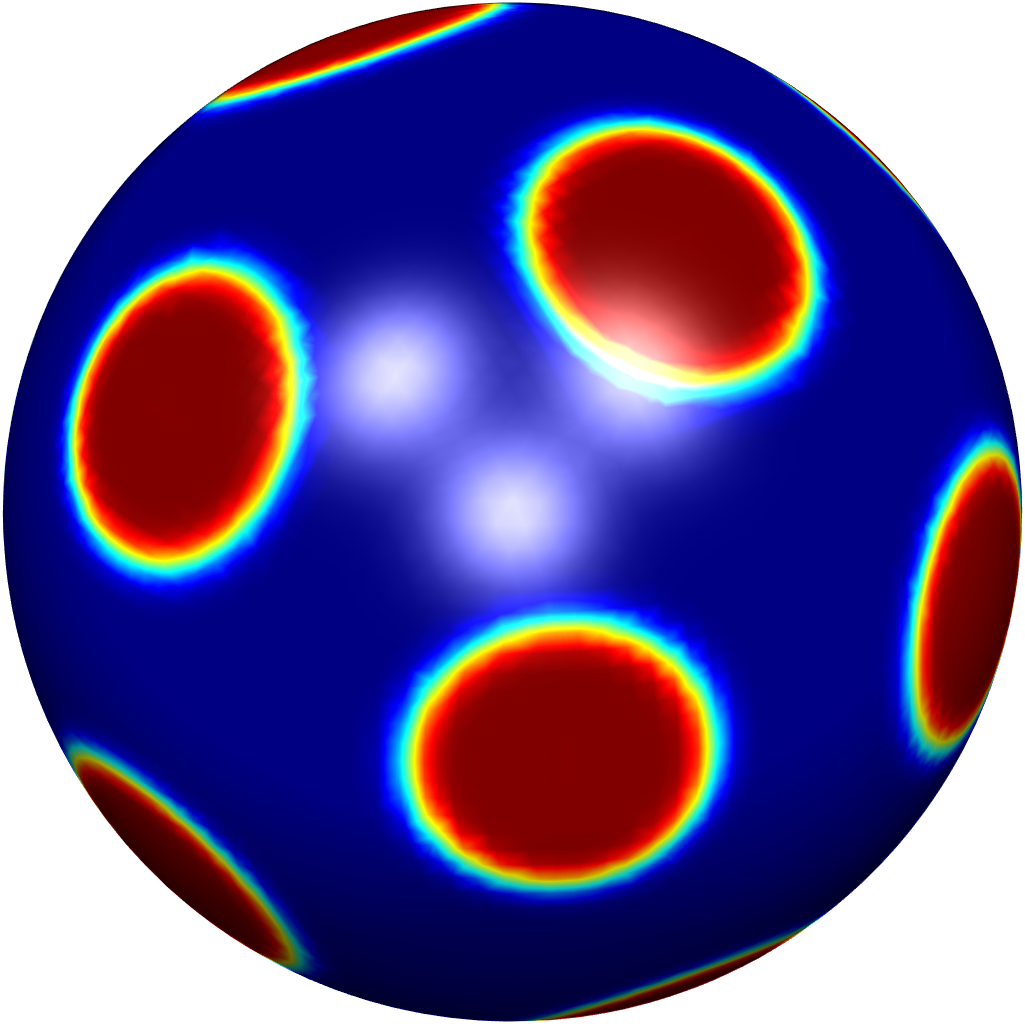} \
    \includegraphics[width=0.18\textwidth] {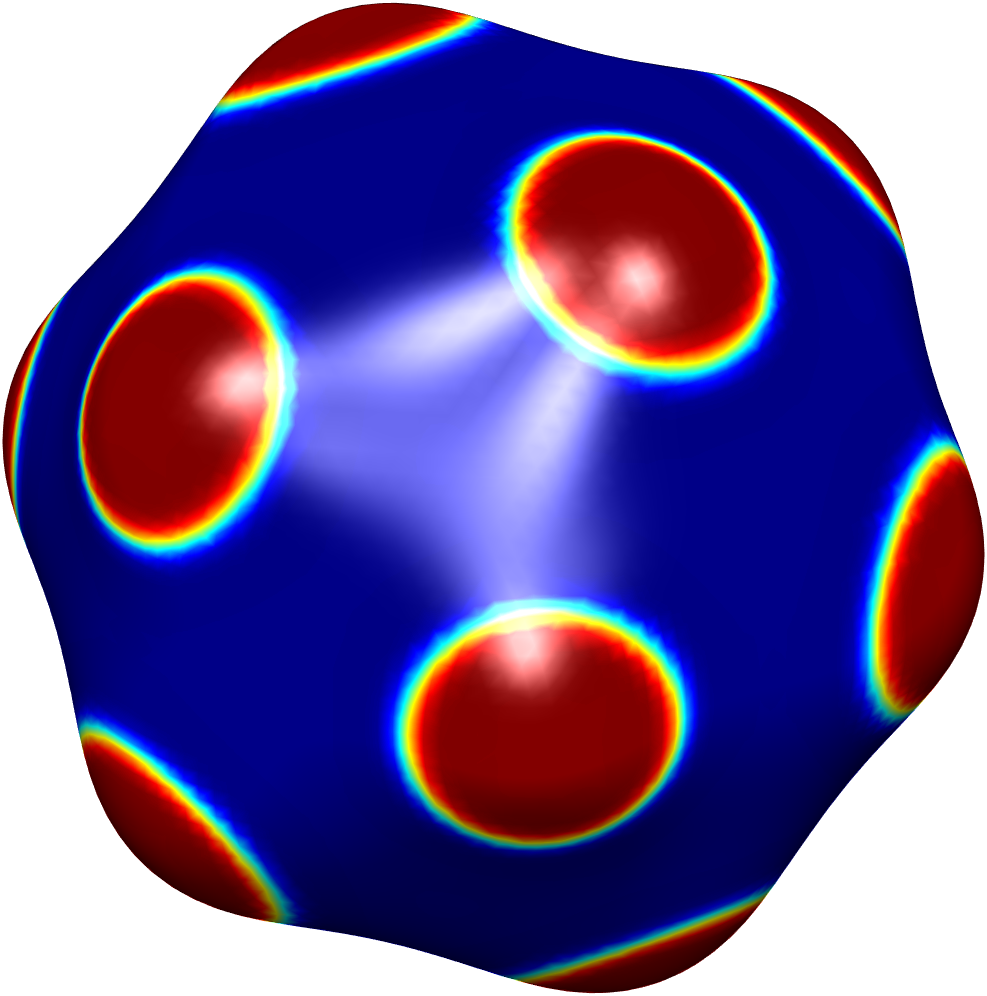} \
    \includegraphics[width=0.18\textwidth]{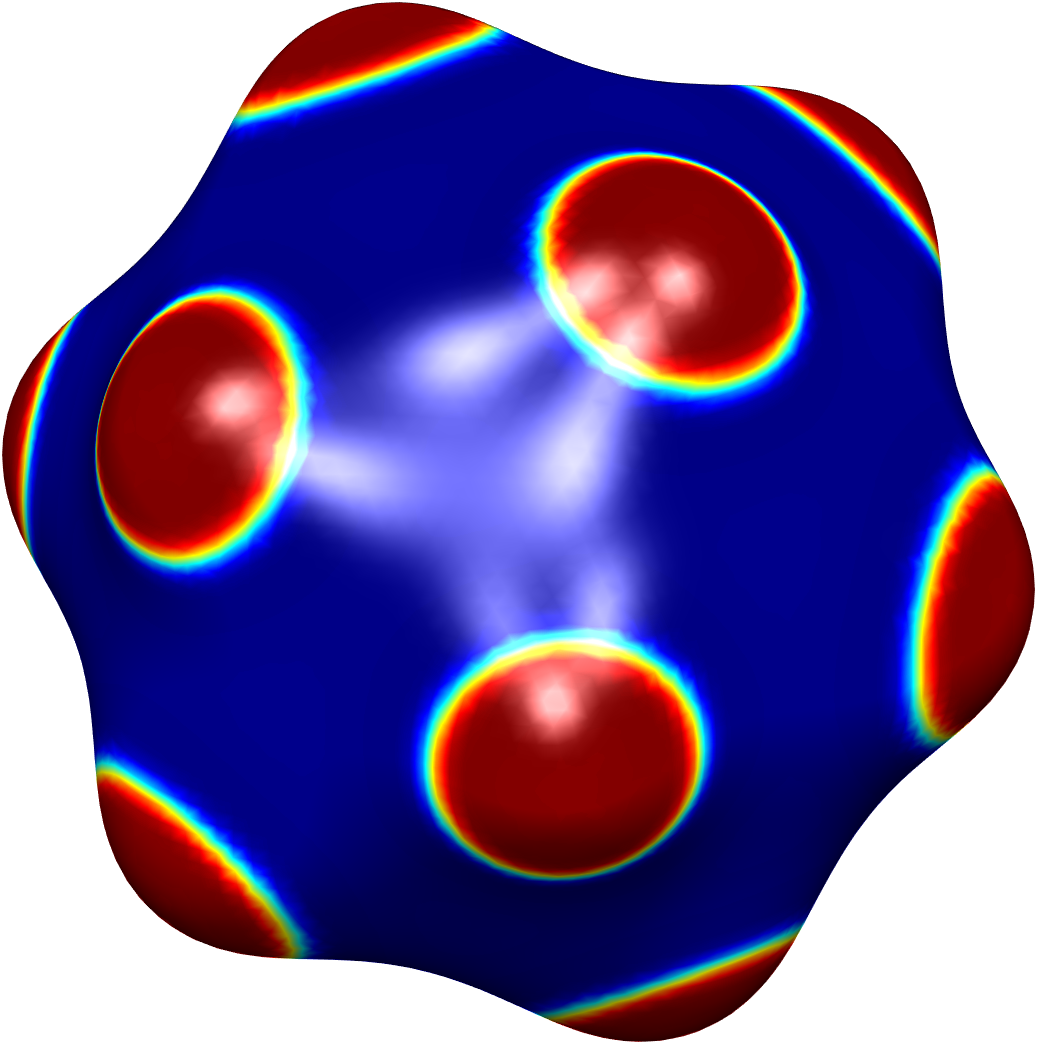} \
    \includegraphics[width=0.18\textwidth]{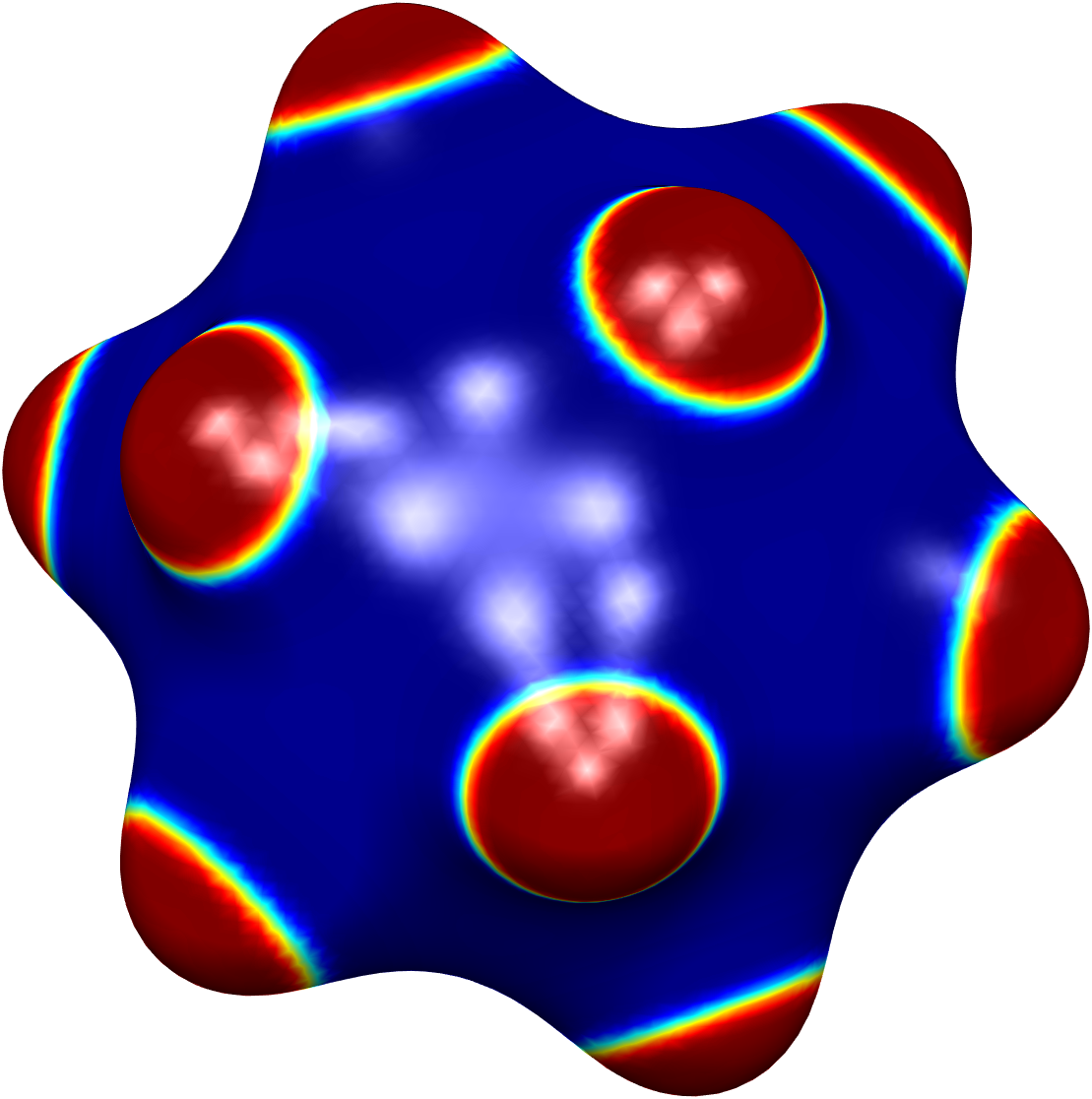} \
    \includegraphics[width=0.18\textwidth]{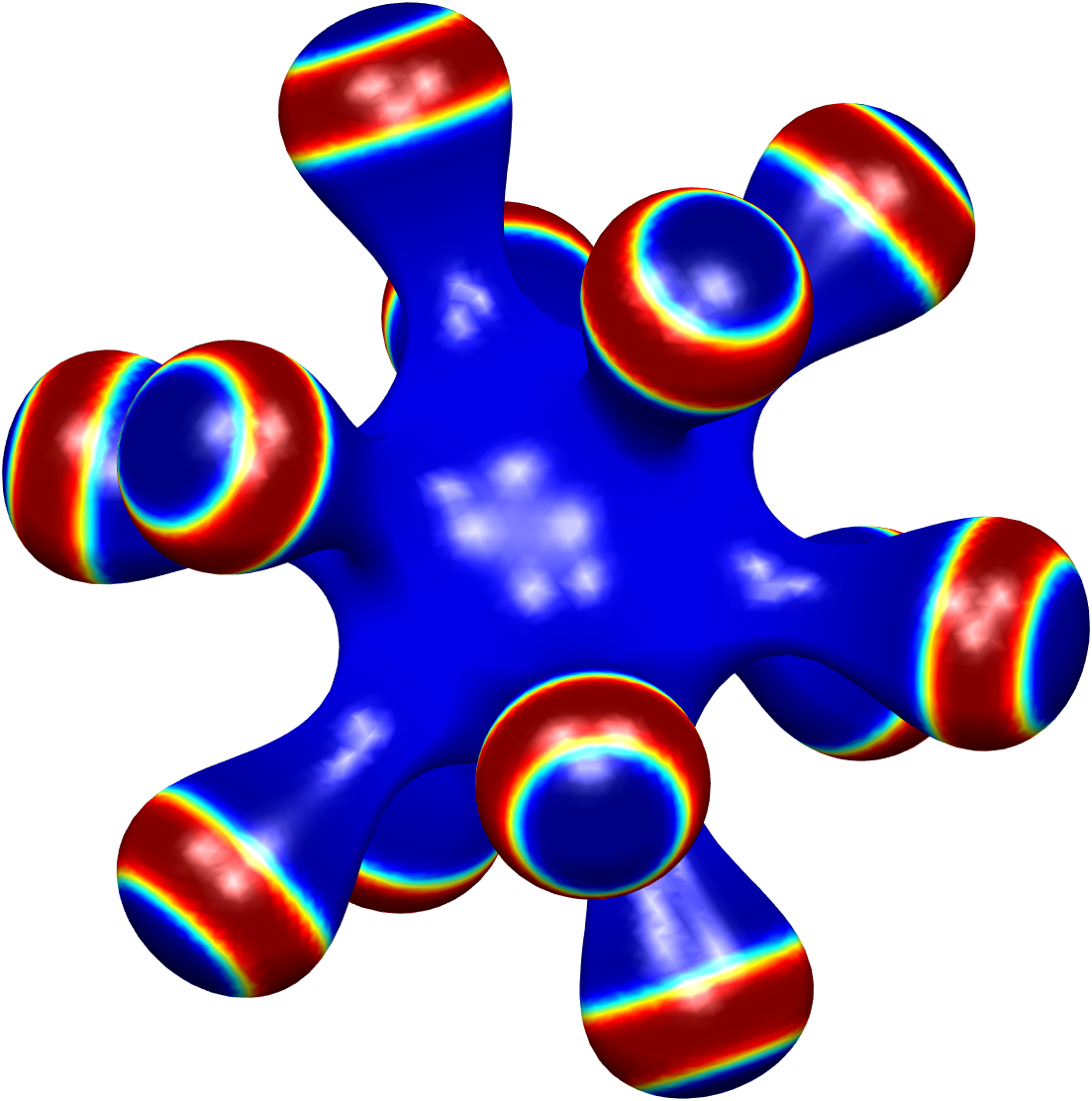}
    \end{center}
    \caption{3D coarsening dynamics process using Algorithm \ref{alg:ETD1_RK2}, parameters $\bar{u} = 0.3$, $\gamma = 4000$, and time step size $\tau = 1e-3$. From left column to right, snapshots for the initial and the biochemical strength $\alpha = 250, 300, 350, 500$. }
    \label{fig:ETDRK2_3d_12b}
\end{figure}

\section{Concluding Remarks}\label{Sec:Conclusion}

In this paper, we first reformulated the OK model on a fixed phase–field membrane and designed first– and second–order ETD schemes combined with a second–order finite difference discretization. We proposed a new operator splitting method based on the localization function, which produces symmetric linear operators and enables rigorous proofs of the discrete maximum bound principle and unconditional energy stability for both ETD1 and ETDRK2 schemes.  To further enhance computational efficiency, a narrow‑band technique was incorporated so that all surface–related quantities are evaluated only in a thin neighborhood of the membrane while preserving accuracy. We then extended the ETD framework to the coupled system of multicomponent membrane \eqref{eqn:model_force}-\eqref{eqn:model_OK} in an FFT-based setting. By using spectral collocation for the force–balance equation and the membrane–associated OK dynamics with suitable stabilizations in time, we obtained highly efficient ETD algorithms for long–time simulations in both two and three dimensions. The numerical experiments support the theoretical results for the fixed–membrane problem and demonstrate rich three–dimensional behaviors such as protein–driven microphase separation, membrane budding, and shape remodeling.
 
Several important questions remain for future investigation. A major direction is to couple the present phase–field formulation with Stokes or Stokes‑type hydrodynamics for the surrounding fluid, which would allow a more realistic treatment of viscous effects and membrane–fluid interactions but introduce significant analytical and numerical challenges. Another direction is to incorporate adhesion forces into the energy and force balance and to carry out one-dimensional mathematical analysis together with systematic two– and three-dimensional simulations using the schemes proposed in this work. These extensions would advance the mechanochemical modeling of multicomponent membranes and facilitate quantitative comparison with experiments on cell–substrate adhesion, vesicle transport, and related processes.


\bibliographystyle{elsarticle-num}
\bibliography{citation}
\end{document}